\documentclass[a4paper,10pt]{article}
\usepackage{a4}
\usepackage{amsfonts}
\usepackage{amsthm}
\usepackage{amssymb}
\usepackage{graphicx, enumerate}
\usepackage{amsmath}
\usepackage{latexsym}
\usepackage{longtable}
\usepackage{tabularx}
\usepackage{enumitem}
\usepackage{amsmath}
\usepackage{amsfonts}
\usepackage{comment}
\usepackage{amsthm}
\usepackage{setspace}
\usepackage{graphicx}
\usepackage{tikz}
\usepackage{hyperref}
\usepackage{url}
\newtheorem{thm}{Theorem}[section]

\newtheorem{prb}[thm]{Problem}
\newtheorem{obs}[thm]{Observation}

\newtheorem{cor}[thm]{Corollary}

\theoremstyle{definition}
\newtheorem{definition}[thm]{Definition}

\newtheorem{exm}[thm]{Example}

 \makeatletter
\def\moverlay{\mathpalette\mov@rlay}
\def\mov@rlay#1#2{\leavevmode\vtop{%
   \baselineskip\z@skip \lineskiplimit-\maxdimen
   \ialign{\hfil$\m@th#1##$\hfil\cr#2\crcr}}}
\newcommand{\charfusion}[3][\mathord]{
    #1{\ifx#1\mathop\vphantom{#2}\fi
        \mathpalette\mov@rlay{#2\cr#3}
      }
    \ifx#1\mathop\expandafter\displaylimits\fi}
\makeatother

\vspace{5cm}

\usepackage{placeins}
\pgfdeclarelayer{nodelayer}
\pgfdeclarelayer{edgelayer}
\pgfsetlayers{edgelayer,nodelayer,main}
\tikzstyle{new style 0}=[fill=blue, draw=black, shape=circle]
\tikzstyle{none}=[fill=none, draw=none, shape=circle]
\tikzstyle{Black Node}=[fill=none, draw=black, shape=circle, text=black, inner sep=2pt]
\tikzstyle{White Node}=[fill=white, draw=none, shape=rectangle, inner sep=0pt]
\tikzstyle{Black Edge}=[-]
\tikzstyle{new edge style 0}=[fill=black, draw={rgb,255: red,191; green,128; blue,64}, -]

\begin{document}

\title
{\Large \sc \bf On some classes of cycles-related $\Gamma$-harmonious graphs}
\date{}
\author{{\sc\bf Gyaneshwar Agrahari$^1$ and Dalibor Froncek$^2$}\\
%{\footnotesize Department of Mathematics and Statistics}\\
{\footnotesize $^1$Louisiana State University, $^2$University of Minnesota Duluth}\\
%{\footnotesize 1117 University Drive}\\
%{\footnotesize Duluth, MN 55812-3000, U.S.A.}\\
{\footnotesize {\it  gyan.agrahari77@gmail.com,  dalibor@d.umn.edu}}
 }
%\end{author}
\maketitle

\begin{abstract}
	A graph $G(V,E)$ is $\Gamma$-harmonious when there is an injection $f$ from $V$ to an Abelian group $\Gamma$ such that the induced edge labels defined as $w(xy)=f(x)+f(y)$ form a bijection from $E$ to $\Gamma$. We study $\Gamma$-harmonious labelings of several cycles-related classes of graphs, including Dutch windmills, generalized prisms, generalized closed and open webs, and superwheels.
\end{abstract}



%\newpage
\section{Introduction}\label{sec:intro}

Harmonious graphs have been extensively studied since 1980, when Graham and Sloane~\cite{Graham-Sloane} first defined the notion. 

A graph $G=(V,E)$ with $q$ edges is \emph{harmonious} if there exists an injection $f$ from the set of vertices of $G$ to the additive group ${Z}_q$, $f: V\to{Z}_q$  such that the induced \emph{edge labels} $w(e)$ of all edges, defined as $w(xy)=(f(x)+f(y))\mod q$, are all distinct. 
That is, the induced function $w:E\to{Z}_q$ is a bijection. When $G$ is a tree, exactly one vertex label is repeated. Graham and Sloane proved that almost all graphs are not harmonious. They also proved that if a harmonious graph has an even number of edges $q$ and the degree of every vertex is divisible by $2k$ then $q$ is divisible by $2k+1$. Liu and Zhang~\cite{Liu-Zhang}  generalized this condition and also proved that there is no forbidden subgraph condition for harmonious graphs, that is, that every graph is a subgraph of a harmonious graph. 

In this paper, we study the concept of $\Gamma$-harmonious graphs by extending the notion of harmonious graphs to any Abelian group $\Gamma$. We say that a graph $G$ with $q$ edges is \emph{$\Gamma$-harmonious} for a given Abelian group $\Gamma$ of order $q$ if there exists an injection $f: V\to\Gamma$ such that the induced labels $w(e)$ of all edges, defined as $w(xy)=f(x)+f(y)$ (where the addition is performed in $\Gamma$), are again all distinct. In other words, the induced function $w:E\to\Gamma$ is a bijection. This notion was introduced by Montgomery, Pokrovskiy and Sudakov in~\cite{MPS}, but their results were all dealing just with cyclic groups. We are not aware of any results for non-cyclic groups whatsoever.

We will present several classes of cycles-related graphs that are $\Gamma$-harmonious for various Abelian groups $\Gamma$. All groups in this paper are finite.

%Auparajita, Dulawat, and Rathore proved that determining whether a graph is harmonious is NP-complete~\cite{Au-Dul-Rat}.

%\newpage
\section{Definitions and notation}\label{sec:defs}

We will use $A, H$ for subgroups, and the group operation will be the usual component-wise addition. A coset determined by a subgroup $A$ of $\Gamma$ and a coset representative $\alpha$ will then be denoted by $\alpha+A$. A cyclic subgroup of $Z_n$ generated by an element $a$ will be denoted by $\langle a\rangle_n$ or simply  $\langle a\rangle$ if no confusion can arise.

First we repeat the formal definition of $\Gamma$-harmonious graphs.

\begin{definition}\label{def:gamma-harmonious}
	 Let $G=(V,E)$ be a graph with $q$ edges and $\Gamma$ an Abelian group of order $q$. We say that $G$ is \emph{$\Gamma$-harmonious} if there exists an injection $f: V\to\Gamma$ (called \emph{$\Gamma$-harmonious labeling}) with the property that the induced labels $w(e)$ of all edges, defined as $w(xy)=f(x)+f(y)$ (where the addition is performed in $\Gamma$), are all distinct. In other words, the induced function $w:E\to\Gamma$ is a bijection. 
	 
	 When $\Gamma=Z_q$, then a $Z_q$-harmonious graph is simply called \emph{harmonious}.
\end{definition}

\begin{definition}\label{def:strongly-harmonious}
     A graph $G$ with $q$ edges is \emph{strongly $\Gamma$-harmonious} if it is $\Gamma$-harmonious for all Abelian groups of order $q$.
\end{definition}

Now we present definitions of the classes of graphs studied in the following sections.

\begin{definition}\label{def:wheel}
	The \emph{wheel graph} $W_n$ arises from the cycle $C_n$ by adding a single vertex and joining it by an edge to every vertex of the cycle.
\end{definition}

	The wheel graph is often described as $C_n + K_1$.

\begin{definition}\label{def:superwheel}
	The \emph{superwheel} $SW_{k,n}$ is a graph arising from the cycle $C_n$ by adding $k$ vertices not adjacent to each other but adjacent to every vertex in $C_n$. 
\end{definition}

	The number of edges in $SW_{k,n}$ is $(k+1)n$. It can be also described as $C_n + kK_1$ or $C_n + \overline{K_k}$. When $k=1$, $SW_{1,n}$ is just the wheel $W_n$.

\begin{definition}\label{def:windmill}
	The \emph{Dutch windmill graph} $D_n^m$ is a graph consisting of $m$ copies of $C_n$ with a common vertex, called the \emph{central vertex}. The $n$-cycles in $D_n^m$ are also called \emph{blades}. 
\end{definition}
	The number of edges of $ D_n^m$ is $nm$. When $m=1$,  $D_n^1$ is a cycle; thus, in our examination, $m>1$. For $n=3$, the graph $D^m_3$ is also called the \emph{friendship graph}. 

\begin{definition}\label{def:gen_prism}
	The \emph{generalized prism} $Y_{m,n}$ is the Cartesian product of the path $P_m$ and the cycle $C_n$.
\end{definition}
	
	The number of edges in $Y_{m,n}$ is $(2m-1)n$. When $m=2$, we obtain the usual \emph{prism}. The cycle $C_n$ can also be seen as the degenerate generalized prism $Y_{1,n}$.

\begin{definition}\label{def:closed_web}
	The \emph{generalized closed web} $CW_{m,n}$ is a graph arising from  the generalized prism $Y_{m,n}$ by adding a new vertex and joining it by an edge to every vertex of degree three of the top cycle of $Y_{m,n}$.
\end{definition}

	The number of edges in $CW_{m,n}$ is $2mn$. The closed web $CW_{2,n}$ is sometimes also called the \emph{closed helm} (see below). The wheel $W_n$ is $CW_{1,n}$. 

\begin{definition}\label{def:open_web}
	The \emph{generalized open web} $OW_{m,n}$ is a graph arising from the generalized closed web $CW_{m,n}$ by removing the edges of the bottom cycle $C_n$.
\end{definition}

		The number of edges in $OW_{m,n}$ is $(2m-1)n$. The \emph{generalized open web} $OW_{m,n}$ is sometimes called just the \emph{generalized web}. The open web $OW_{2,n}$ is better know as the \emph{helm graph}. The usual \emph{web graph} is then $OW_{3,n}$.

%\newpage
\section{Harmonious groups}\label{sec:related}

For the sake of completeness, we first state some well known group theory results, which we will use later, starting with \emph{The Fundamental Theorem of Finite Abelian Groups.} 

%\marbl{Re-phrase maybe}
\begin{thm}[The Fundamental Theorem of Finite Abelian Groups]\label{thm:fund-Abel}
	Let $\Gamma$ be an Abelian group of order $n=p^{s_1}_1 p^{s_2}_2\dots p^{s_k}_k$, where
	$k\geq1$,  $p_1,p_2,\dots,p_k$ are primes, not necessarily distinct, and $s_1,s_2,\dots,s_k$ positive integers.
	
	Then $\Gamma$ is isomorphic to
	$Z_{p^{s_1}_1}\oplus Z_{p^{s_2}_2}\oplus\dots\oplus Z_{p^{s_k}_k}$ and if we moreover require that
	 for every $i=1,2,\dots,k-1$ we have $p_i\leq p_{i+1}$ and if $p_i=p_{i+1}$, then $s_i\leq s_{i+1}$, then the expression determined by 
	the $k$-tuple $(p^{s_1}_1,p^{s_2}_2,\dots,p^{s_k}_k)$ is unique and we call it the \emph{canonical form} of the group $\Gamma$. 
	%and denote by $\Gamma(p^{s_1}_1,p^{s_2}_2,\dots,p^{s_k}_k)$.
\end{thm}

Sometimes it is useful to express an Abelian group in a different way.

\begin{thm}%[Abelian Groups]
	\label{thm:Abel-chain}
	Let  $\Gamma=Z_{p^{s_1}_1}\oplus Z_{p^{s_2}_2}\oplus\dots\oplus Z_{p^{s_k}_k}$ of order $n$ be as in Theorem~\ref{thm:fund-Abel}.
	Then it can be written in a unique way in \emph{nested form} as
	$\Gamma=Z_{n_1}\oplus Z_{n_2}\oplus\dots\oplus Z_{n_t}$, where 
	$n_t\geq2, \ n= n_1 n_2\dots n_t,$ and $n_{i+1}\mid n_i$ for $i=1,2,\dots,t-1$.
\end{thm}

Another well known theorem is the following.

\begin{thm}\label{thm:cyclic-group}
	An Abelian group $\Gamma%(p^{s_1}_1,p^{s_2}_2,\dots,p^{s_k}_k)
	=Z_{p^{s_1}_1}\oplus Z_{p^{s_2}_2}\oplus\dots\oplus Z_{p^{s_k}_k}$ is cyclic if and only if all primes $p_i$, $i=1,2,\dots,k$ are distinct.
\end{thm}

The following observation will be useful.

\begin{obs}\label{obs:cyclic subgroup}
	Let $\Gamma=Z_{n_1}\oplus Z_{n_2}\oplus\dots\oplus Z_{n_t}$ of order $n$ satisfy conclusion of Theorem~\ref{thm:Abel-chain} and $H$ be its cyclic subgroup of order $m$. Then $m\mid n_1$ and $Z_{n_1}$ has a cyclic subgroup of order $m$.
\end{obs}

\begin{proof}
	Let $Z_{q^{z_1}_1}\oplus Z_{q^{z_2}_2}\oplus\dots\oplus Z_{q^{z_r}_r}$ be the canonical form of $Z_{n_1}$. Then by Theorem~\ref{thm:cyclic-group} $q_1,q_2,\dots,q_r$ are all distinct. Moreover, $q_1,q_2,\dots,q_r$ are all  primes in the multiset $\{p_1,p_2,\dots,p_k\}$ defining the canonical form of $\Gamma$. For if not, then there is $p_i$ which appears in the canonical form of some $Z_{n_j}, j>1$ but not in $\{q_1,q_2,\dots,q_r\}$. But then $n_j$ does not divide $n_1$, which contradicts Theorem~\ref{thm:Abel-chain}. Therefore, we can write $H$ as 
	$Z_{q^{u_1}_1}\oplus Z_{q^{u_2}_2}\oplus\dots\oplus Z_{q^{u_r}_r}$, where we allow each $u_i$ to be a non-negative integer.
	
	For the same reason as above we must have $u_i\leq z_i$ for every $i=1,2,\dots,r$, since otherwise $q^{u_i}_i$ appears in the canonical form of some $Z_{n_j}, j>1$ and $n_j\nmid n_1$, a contradiction.
\end{proof}

%\newpage

Another well-known result will be useful.

\begin{thm}\label{thm:fact-group-isom-to-subgroup}
     Let $\Gamma$ be a finite Abelian group. Then
     
	\begin{itemize}
	\item[(1)]
	if $H$ is a subgroup of $\Gamma$, then the quotient group $\Gamma/H$ is isomorphic to some subgroup $K$ of\/ $\Gamma$, and conversely,
	\item[(2)]
	if $K$ is a subgroup of $\Gamma$, then there exists a subgroup $H$ such that $K\cap H=\{e\}$ and the quotient group $\Gamma/H$ is isomorphic to $K$.     
	\end{itemize}
\end{thm}

The main building block we will be using in our constructions is a theorem that follows as a direct corollary form a group theory result proved by Beals, Gallian, Headly and Jungreis~\cite{BGHJ}. They studied {harmonious groups.}

\begin{definition}\label{def:harm-group}
	A finite group $\Gamma$ of order $n$ is \emph{harmonious} if the elements can be ordered  $(g_1,g_2,\dots,g_n)$ so that the set of element products $\{g_1g_2,g_2g_3,\dots,g_{n-1}g_n,g_ng_1\}$ is equal to the set of all elements of $\Gamma$. The ordered list $(g_1,g_2,\dots,g_n)$ is called the \emph{harmonious sequence} of $\Gamma$.

\end{definition}

They proved the following characterization of harmonious groups.

\begin{thm}[Beals et al.~\cite{BGHJ}]
	\label{thm:harm-groups}
	If\/ $\Gamma$ is a finite, non-trivial Abelian group, then\/ $\Gamma$ is har\-mo\-nious if and only $\Gamma$ has a non-cyclic or trivial Sylow $2$-subgroup and $\Gamma$ is not an elementary $2$-group.
\end{thm}

%	In other words, all non-trivial Abelian groups of odd order are har\-mo\-nious, and groups of even order \bblue{must contain a subgroup $Z_{2^a}\oplus Z_{2^b}, a,b\geq 1$} but cannot be equal to $Z_2\oplus Z_2\oplus\dots\oplus Z_2$.
	
	In other words, all non-trivial Abelian groups of odd order are har\-mo\-nious, and a harmonious group of even order {must contain a subgroup isomorphic to $Z_{2}\oplus Z_{2}$} but cannot be equal to $Z_2\oplus Z_2\oplus\dots\oplus Z_2$. That is, it must contain a subgroup isomorphic to $Z_2\oplus Z_{2s}$ where $s\geq2$.

An immediate consequence of Theorem~\ref{thm:harm-groups} is a result on $\Gamma$-harmonious labeling of cycles, which we state in the next section.

%\newpage
\section{Known results}\label{sec:known}

The first result on cycle-related harmonious graphs was proved by Graham and Sloane~\cite{Graham-Sloane}. Note that ``harmonious'' here means ``$Z_n$-harmonious.''

\begin{thm}[Graham and Sloane~\cite{Graham-Sloane}]
	\label{thm:GS-odd-cycles}
	The cycle $C_n$ with $n\geq 3$ is harmonious if and only if $n$ is odd.
\end{thm}

A more general theorem is a direct consequence of Theorem~\ref{thm:harm-groups}. Although it is technically speaking a new result, as it was not formally stated and proved in~\cite{BGHJ}, we state it here because we believe that the authors were aware of it.

\begin{thm}
\label{thm:all-cycle}
  The cycle $C_n$ is $\Gamma$-harmonious if and only if $\Gamma$ is of order $n$,
   and 
  \begin{enumerate}
      \item $n$ is odd, or
      \item $n$ is even and $\Gamma$ has a subgroup isomorphic to $Z_2\oplus Z_{2s}$ where $s\geq2$. 
      % but
      % $\Gamma\neq Z_2\oplus Z_2\oplus\dots\oplus Z_2$.
  \end{enumerate}
\end{thm}
\begin{proof}
   This follows directly from Theorem~\ref{thm:harm-groups}. Label the vertices consecutively with the elements of $\Gamma$ in their harmonious sequence order.
\end{proof}

Graham and Sloane~\cite{Graham-Sloane} further proved the following result on wheel graphs.
\begin{thm}[Graham and Sloane~\cite{Graham-Sloane}]
\label{thm:GS-wheels}
    All wheels $W_n$ are harmonious.
\end{thm}

We generalize their result in the  Section~\ref{sec:prisms}. We prove that odd wheels $W_{2k+1}$ are strongly $\Gamma$-harmonious, and even wheels are $\Gamma$-harmonious for certain groups $\Gamma$.

Gallian~\cite{Gallian} cites~\cite{Gnanajothi} which has the following two results on superwheel graphs.
% \marred{You mentioned ``two results'' in your write-up, but only stated one. Can you please check what is correct?}
\begin{thm}[{Gnanajothi}~\cite{Gnanajothi}]
    The superwheel graph $S_{k,n}$ is harmonious if $n$ is odd and $k=2$. 
    % and not harmonious if 
    % $n\equiv 2,4,6 \mod 8 $ and $k=2$.
\end{thm}
\begin{thm}[{Gnanajothi}~\cite{Gnanajothi}]
The superwheel graph $S_{k,n}$ is not harmonious if $n \equiv 2,4,6 \mod 8$ and $k=2$.  
\end{thm}
We will show that all odd superwheel graphs are strongly $\Gamma$-harmonious and even superwheel graphs are $\Gamma$-harmonious for some groups.

Graham and Sloane~\cite{Graham-Sloane} also established the following result on Dutch windmill graphs.

%\marred{They had if and only if}
\begin{thm}[Graham and Sloane~\cite{Graham-Sloane}]
\label{thm:GS-windmill}
 The Dutch windmill graph $D_3^m$  is harmonious if and only if $ m\not \equiv 2 \mod 4$.
\end{thm}

We will prove that when both $m$ and $n $ are odd, $D_n^m$ is strongly $\Gamma$-harmonious. Moreover, we will show for $m$ odd and $n$ even, $D_n^m$ is $\Gamma$-harmonious for certain groups. 
 
Also, Graham and Sloane~\cite{Graham-Sloane} proved the following result on generalized prisms.
 
\begin{thm}[Graham and Sloane~\cite{Graham-Sloane}]
\label{thm:GS-gen-prisms}
    The generalized prism $Y_{m,n}$ is harmonious whenever $n$ is odd and $m\geq 2$.\\
%    \rred{They also proved this result to be true for $m=2$; Gallian's survey I have has this result cited incorrectly; the survey says the generalized prism is harmonious when the length of the path is odd; whereas the paper says they are harmonious when the cycles are odd. Jungreis and Rei also had more results than the survey cited it for.}
\end{thm}
%\marbl{I added Headley here and elsewhere. My reprint shows four co-authors.}

%\marbl{Can you please verify this and check whether there are any more ``prism-related'' results we should cite?\\
%This is done.}
Later,  Gallian, Prout and Winters~\cite{GPW} showed the following result on even prisms.

\begin{thm}[Gallian et al.~\cite{GPW}]
\label{thm:BGHJ-prisms}
    The prism $Y_{2,n}$ is harmonious except when $n=4$.
\end{thm}
 Jungreis and Reid ~\cite{Jungreis-Reid} showed the following results on type harmonious prisms. $Y_{2m,4n}$ and $Y_{2m+1,4m}$
\begin{thm}[Jungreis and Reid~\cite{Jungreis-Reid}]
    The generalized prism $Y_{m,n}$ is harmonious if:
     \begin{enumerate}
         \item $m=2k$ and $n=4l $ except when $ (k,l)=(1,1)$
         \item $m=2k+1$ and $n=4l$
         \item $m=2k$ and $n=4l+2$
         \item $ n=2l+1$
     \end{enumerate}
\end{thm}
 
% \begin{thm}[Jungreis and Reid~\cite{Jungreis-Reid}]
%   The generalized prism, $Y_{2m,4n}$ is harmonious except when $(m,n)=(1,1)$ and $Y_{2m+1,4n}$ is harmonious for all $m,n\geq 1$.
% \end{thm}
%  Additonally, Jungreis and Reid ~\cite{Jungreis-Reid} proved the following result on $Y_{m,4n+2}$.
% \begin{thm}[Jungreis and Reid~\cite{Jungreis-Reid}]
%     The generalized prism $Y_{m,4n+2}$ is harmonious when $m$ is even and $n\geq 1$.
% \end{thm}
We prove results on generalized prisms showing $Y_{m,n}$ is $\Gamma$-harmonious for some $\Gamma$. 

Seoud and Youssef~\cite{Seoud-Youssef} proved the following two results on open webs (or helm graphs) $OW_{2,n}$ and closed webs (or closed helms) $CW_{2,n}$. In~\cite{Gallian}, Gallian cites~\cite{Gnanajothi} which states the following result on open webs $OW_{3,n}$.

\begin{thm}[Seoud and Youssef~\cite{Seoud-Youssef}]
\label{thm:SY-open webs}
  The open web $OW_{2,n}$ is harmonious for all $n\geq 3$.
\end{thm}

\begin{thm}[Gnanajothi~\cite{Gnanajothi}]
\label{thm:Gna-open webs}
The open web $OW_{3,n}$ is harmonious when $n$ is odd and $n\geq 3$.
\end{thm}

\begin{thm}[Seoud and Youssef~\cite{Seoud-Youssef}]
\label{thm:SY-closed webs}
  The closed web $CW_{2,n}$ is harmonious for all odd $n\geq 3$. 
\end{thm}

We will generalize results on $OW(m,n)$ and $CW(m,n)$ by showing that they are $\Gamma$-harmonious for certain $\Gamma$.

%\newpage
\section{Superwheels}\label{sec:superwheels}

The superwheel $SW_{k,n}$ is a graph with a cycle of length $C_n$ and has $k$ vertices not adjacent to each other but all  adjacent to every vertex in $C_n$. The number of edges of $SW_{k,n}$ is $n(k+1)$. When $k=1$, then $SW_{k,n}$ is the wheel $W_n$.

\begin{thm}\label{thm:superwheel}
	The superwheel $SW_{k,n}$ is $\Gamma$-harmonious with $\Gamma$ of order $n(k+1)$ if there exists a subgroup $H$ of order $n$ that satisfies conditions in Theorem~\ref{thm:all-cycle}.
\end{thm}

\begin{proof} 
    Let $v_1,v_2,\dots,v_n$ be the vertices of the $C_n$ in $SW_{k,n}$ and $u_1,u_2,\dots,u_k$ be the the vertices in its center.
    For $ 1\leq j\leq k$ let $E_j=\{u_jv_i|v_i\in V(C_n)\}$. 

	Assume $H$ is a subgroup of order $n$ that satisfies Theorem \ref{thm:all-cycle}. Then we can find a $\Gamma$-harmonious labeling of $C_n$ with $ H$.
    Let $ (h_1,h_2,\dots,h_n)$ be a $H$-harmonious labeling of $C_n$, where for $1\leq i\leq n, v_i$ is labeled with the element $h_i\in H$. We can write $\Gamma$ as the following disjoint union, $\Gamma= H\cup g_1+H\cup g_2+H\cup\dots\cup g_k+H$ for $ g_1,g_2,\dots,g_k$ in $G- H$. For $1\leq j\leq k$, label the central vertex $ u_j$ with $g_j$. Thus, for a fixed $j$, the edge $u_jv_i$ in $E_j$ would be labeled with $g_j+h_i$, which is an element of $g_j +H$.
    
    Since each $ g_j+h_i$ are distinct, there would be no repetition among the edge labels of the edges in $E_j$. 
    Thus, the edges of $C_n$ have labels equal to elements of $H$ and the edges of $E_j$ are labeled with those of $g_j+H$. Since the cosets are disjoint, there would be no repetition in the edge labelings across all edges of $SW_{k,n}$
    Thus, we have our $ \Gamma$-Harmonious labeling of $ SW_{k,n}$.
\end{proof}

We illustrate the method by two examples.

%%%%%%%%%%%%%%%%%%%%%%%%%%%%
%%%%%%%%%%%%%%%%%%%%%%%%%%%%
%%%%%%%%%%%%%%%%%%%%%%%%%%%%
%%%%%%%%%%%%%%%%%%%%%%%%%%%%
%\newpage
\begin{exm}\label{exm:SW_3,5}
	A $Z_5\oplus Z_5$ labeling of the superwheel $SW_{3,5}$.
\end{exm}

\begin{figure}[h!]\label{fig:SW_{3,5}}
	\centering
	
	\scalebox{0.55}{
\begin{tikzpicture}
    % Labels for the path P vertices
    \newcommand{\pathLabels}[1]{%
        \ifcase#1 (0,0)\or (2,0)\or (4,0)\or (6,0)\or (8,0)\fi
    }
    \newcommand{\CenterLabels}[1]{%
        \ifcase#1 (1,0)\or (1,1)\or (0,1)\fi
    }

    % Draw the path P with 5 vertices and custom labels
    \foreach \i in {0,...,4} {
        \node[draw, circle, inner sep=2pt, label=below:{\pathLabels{\i}}] (P\i) at (2*\i,0) {};
    }
    \draw (P0) -- (P1) -- (P2) -- (P3) -- (P4);

    % Draw the concave curve connecting the ends of the path
    \draw (P0) to[bend right=35] (P4);

    % Add vertices V1, V2, V3, 2cm above the path horizontally
    \foreach \i in {0,1,2} {
        \node[draw, circle, inner sep=2pt, label=above:{\CenterLabels{\i}}] (v\i) at (2.5*\i+1,3) {};
    }

    % Connect each V_i to all the vertices of P
    \foreach \i in {0,1,2} {
        \foreach \j in {0,...,4} {
            \draw (v\i) -- (P\j);
        }
    }
\end{tikzpicture}

} 
	\caption{Labeling of $SW_{3,5} $ with $Z_{10}\oplus  Z_2$}
\end{figure}
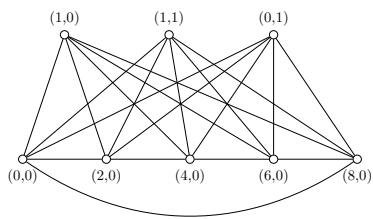
%%%%%%%%%%%%%%%%%%%%%%%%%%%%
\newpage
\begin{exm}\label{exm:SW_3,8}
	A $Z_4\oplus Z_2\oplus Z_4$ labeling of the superwheel $SW_{3,8}$.
\end{exm}

\begin{figure}[h!]\label{fig:SW_{3,8}}
	\centering
	
	\scalebox{0.55}{
\begin{tikzpicture}
    % Labels for the path P vertices
    \newcommand{\pathLabels}[1]{%
        \ifcase#1 (0,0,0)\or (2,0,0)\or (1,0,0)\or (2,1,0)\or (3,0,0)\or (3,1,0)\or (1,1,0)\or (0,1,0) \fi
    }
    \newcommand{\CenterLabels}[1]{%
        \ifcase#1 (0,0,1)\or (0,0,2)\or (0,0,3)\fi
    }

    % Draw the path P with 5 vertices and custom labels
    \foreach \i in {0,...,7} {
        \node[draw, circle,  inner sep=2pt, label=below:{\pathLabels{\i}}] (P\i) at (2*\i,0) {};
    }
    \draw (P0) -- (P1) -- (P2) -- (P3) -- (P4)--(P5)--(P6)--(P7);

    % Draw the concave curve connecting the ends of the path
    \draw (P0) to[bend right=27] (P7);

    % Add vertices V1, V2, V3, 2cm above the path horizontally
    \foreach \i in {0,1,2} {
        \node[draw, circle, inner sep=2pt, label=above:{\CenterLabels{\i}}] (v\i) at (3.5*\i+3,3) {};
    }

    % Connect each V_i to all the vertices of P
    \foreach \i in {0,1,2} {
        \foreach \j in {0,...,4,5,6,7} {
            \draw (v\i) -- (P\j);
        }
    }
\end{tikzpicture}

} 
	\caption{Labeling of $SW_{3,5} $ with $Z_4\oplus  Z_2\oplus Z_4$}
\end{figure}
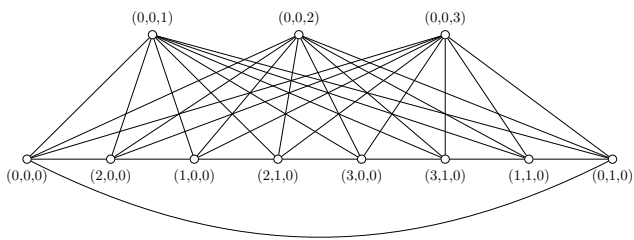
%%%%%%%%%%%%%%%%%%%%%%%%%%%%
%%%%%%%%%%%%%%%%%%%%%%%%%%%%
%%%%%%%%%%%%%%%%%%%%%%%%%%%%
%%%%%%%%%%%%%%%%%%%%%%%%%%%%

%\newpage
\section{Windmill graphs}\label{sec:windmills}

% Node styles
\tikzstyle{none}=[fill=none, shape=circle]
\tikzstyle{new style 1}=[fill=none, draw=black, shape=circle]

% Edge styles
\tikzstyle{new edge style 1}=[fill=black, draw={rgb,255: red,191; green,128; blue,64}, -]

The \emph{Dutch windmill graph} $D_n^m$ is a graph consisting of $m$ copies of $C_n$ with a common vertex, called the \emph{central vertex}. 
The $n$-cycles in $D_n^m$ are also called \emph{blades}. 
Therefore, the number of edges of $ D_n^m$ is $nm$.
When $m=1$,  $D_n^m$ is a cycle; thus, in our examination, $m>1$.

We know that $m\mid|\Gamma|$. Therefore, there must be a subgroup $H$ of order $m$ and index $n$. Because $\Gamma$ is Abelian, $H$ is normal. Therefore, $K=\Gamma/H$ is a quotient group. Denote the coset representatives of $K$ by 
$\gamma_{0}=\mathbf{0}, \gamma_{1},\gamma_{2},\dots,\gamma_{n-1}$, where $\mathbf{0}$ is the identity element of $\Gamma$ (and thus of $H$ as well). Since there are $m$ copies of $C_n$, we can denote each $C_n$ as $C^j$ where $0\leq j\leq m-1$. Let $u$ be the central vertex. In each $C^j$, we denote other vertices by $v^j_i$ for $i=1,2,\dots,{n-1}$ such that the the edges in $C^j$ are $e^j_0=uv^j_1, e^j_i=v^j_i v^j_{i+1}$ for $i=1,2,\dots, n-2$, and $e^j_{n-1}=v^j_{n-1}u$.

Before we start constructing the vertex labeling of $D^m_n$ with $\Gamma$, it is convenient to prove the following observation. We will show that there exists a valid labeling of $C^j$ with $K$. The coset representatives of the elements of $K$ in the vertex labeling of $v_i^j$ in the $K$-harmonious labeling of $C^j$ will be later used in a construction of a valid labeling of $D_n^m$ with $\Gamma$.

\begin{obs}
\label{obs:blade labeling odd}
    If $n$ is odd then each $C^j$ has a $K$-harmonious labeling where $u$ is labeled with $H$.
\end{obs}
\begin{proof}
    Since $K$ is an Abelian group of order $n$, there exists a $K$-harmonious labeling of each $C^j$. 
    Thus, we can find a $K$-harmonious labeling $g$ of $C^j$ such that $u$ is labeled with $H$ and $v_i^j$ is labeled with $g(v^j_i)=\gamma_i+H$ for $\gamma_i+H\in K=\Gamma/H$, where $\gamma_{i}\notin H$. The induced edge label of edge  $e$ is denoted by $w_g(e)$.
    Notice that while we are labeling the cycle with cosets, the labeling is well-defined since they are elements of the group $K$. 
 
% \vskip10pt
% \marbl{Do we really need this paragraph? \textbf{Please check that out and let me know.}}  

\
  Let $ \Tilde{\gamma}_0+H=\gamma_1+H$ and $\Tilde{\gamma}_{n-1}+H=\gamma_{n-1}+H$
    and for $1\leq i\leq n-2$,
     $\Tilde{\gamma}_{i}+H=\gamma_i+\gamma_{i+1}+H$.
    In this scheme, $w(e_i^j)=\Tilde{\gamma}_i+H$. 
%    
    % $e_1^j,e_2^j,...,e_{n-2}^j,e_{n-1}^j$ are labeled with $ \Tilde{\gamma}_0+H,\Tilde{\gamma}_1+H,...,\Tilde{\gamma}_{n-2}+H,\Tilde{\gamma}_{n-1}+H$ respectively. 
  \end{proof}    
Let us look at a specific example how we would use this lemma to construct a vertex labeling of $D_5^3$ with $\Gamma={Z}_5\oplus {Z}_3$ that would induce a $\Gamma$-harmonious labeling.

Here, we have, $ H=\{0\}\oplus{Z}_3$ and $K=\{H,(1,0)+H,(2,0)+H,(3,0)+H,(4,0)+H\}$.
In Figure \ref{D_5^3graph} (a),  we have  a $K$-harmonious labeling for each $C^j$, where $u$ is labeled with $H$, and in each $ C^j$, $ v_i^j$ is labeled with $ (i,0)+H$.
 In Figure \ref{D_5^3graph}(b), $u$ is labeled with $(0,0)$. We have chosen a distinct element of $\{0\}\oplus {Z}_3$ for each $C^j$ specifically, $h_1=(0,0)$ for $C^1$, $h_2=(0,1)$ for $C^2$, and $h_3=(0,2)$ for $C^3$. 
We labeled the vertices $v_i^1$ with $ (1,0)+(0,0),(2,0)+(0,0),(3,0)+(0,0),(4,0)+(0,0)$.
Similarly, we labeled the vertices $v_i^2$ with $\{i,0)+(0,1)\}$ and vertices $v_i^3$ with $\{i,0)+(0,2)\} $. We can check that this construction produces a $\Gamma$-harmonious labeling of $D_5^3$.

\begin{figure}[h!]\label{fig:windmill D^3_5}
    \centering
    
	\scalebox{0.55}{
  %  \resizebox{{5cm}}{!}
    \begin{tikzpicture}
    	\begin{pgfonlayer}{nodelayer}
		\node [style=new style 1] (0) at (-1.75, 0) {};
		\node [style=new style 1] (1) at (-3.25, 2) {};
		\node [style=new style 1] (2) at (-0.25, 2) {};
		\node [style=new style 1] (3) at (-3.25, 4) {};
		\node [style=new style 1] (4) at (-0.25, 4) {};
		\node [style=new style 1] (13) at (-4.25, 0.75) {};
		\node [style=new style 1] (14) at (-4.25, -1.75) {};
		\node [style=new style 1] (15) at (-6.25, 0.75) {};
		\node [style=new style 1] (16) at (-6.25, -1.75) {};
		\node [style=new style 1] (17) at (0.75, 0.75) {};
		\node [style=new style 1] (18) at (0.75, -1.75) {};
		\node [style=new style 1] (19) at (2.75, 0.75) {};
		\node [style=new style 1] (20) at (2.75, -1.75) {};
		\node [style=none] (21) at (-1.75, -1) {};
		\node [style=none] (22) at (-4.25, -2.5) {$(4,0)+H$};
		\node [style=none] (23) at (-6.25, -2.5) {$(3,0)+H$};
		\node [style=none] (24) at (-4.25, 2) {$(4,0)+H$};
		\node [style=none] (25) at (-4.25, 0) {$(1,0)+H$};
		\node [style=none] (26) at (0.75, 0) {$(4,0)+H$};
		\node [style=none] (27) at (-3.25, 4.5) {$(3,0)+H$};
		\node [style=none] (28) at (-0.25, 5) {};
		\node [style=none] (29) at (0.75, 2) {};
		\node [style=none] (30) at (0.75, -2.5) {$(1,0)+H$};
		\node [style=none] (31) at (2.75, -2.5) {$(2,0)+H$};
		\node [style=none] (32) at (-7.25, 0.75) {$(2,0)+H$};
		\node [style=none] (33) at (3.75, 0.75) {$(3,0)+H$};
		\node [style=none] (34) at (-1.75, -1) {$H$};
		\node [style=none] (35) at (0.75, 2) {$(1,0)+H$};
		\node [style=none] (36) at (-0.25, 4.5) {$(2,0)+H$};
		\node [style=new style 1] (37) at (-1.5, -10.25) {};
		\node [style=new style 1] (38) at (-3, -8.25) {};
		\node [style=new style 1] (39) at (0, -8.25) {};
		\node [style=new style 1] (40) at (-3, -6.25) {};
		\node [style=new style 1] (41) at (0, -6.25) {};
		\node [style=new style 1] (42) at (-4, -9.5) {};
		\node [style=new style 1] (43) at (-4, -12) {};
		\node [style=new style 1] (44) at (-6, -9.5) {};
		\node [style=new style 1] (45) at (-6, -12) {};
		\node [style=new style 1] (46) at (1, -9.5) {};
		\node [style=new style 1] (47) at (1, -12) {};
		\node [style=new style 1] (48) at (3, -9.5) {};
		\node [style=new style 1] (49) at (3, -12) {};
		\node [style=none] (50) at (-1.5, -11.25) {};
		\node [style=none] (51) at (-4, -12.75) {$(4,1)$};
		\node [style=none] (52) at (-6, -12.75) {$(3,1)$};
		\node [style=none] (53) at (-4, -8.25) {$(4,0)$};
		\node [style=none] (54) at (-4, -10.25) {$(1,1)$};
		\node [style=none] (55) at (1, -10.25) {$(4,2)$};
		\node [style=none] (56) at (-3, -5.75) {$(3,0)$};
		\node [style=none] (57) at (0, -5.25) {};
		\node [style=none] (58) at (1, -8.25) {};
		\node [style=none] (59) at (1, -12.75) {$(1,2)$};
		\node [style=none] (60) at (3, -12.75) {$(2,2)$};
		\node [style=none] (61) at (-7, -9.25) {$(2,1)$};
		\node [style=none] (62) at (4, -9.5) {$(3,2)$};
		\node [style=none] (63) at (-1.5, -11.25) {$(0,0)$};
		\node [style=none] (64) at (1, -8.25) {$(1,0)$};
		\node [style=none] (65) at (0, -5.75) {$(2,0)$};
		\node [style=none] (66) at (-1.75, 3) {$C^1$};
		\node [style=none] (67) at (-5.5, -0.5) {$C^2$};
		\node [style=none] (68) at (1.75, -0.5) {$C^3$};
		\node [style=none] (69) at (-1.5, -7) {$C^1$};
		\node [style=none] (70) at (-5.25, -10.75) {$C^2$};
		\node [style=none] (71) at (2, -10.75) {$C^3$};
		\node [style=none] (72) at (-1.5, -4) {};
		\node [style=none] (73) at (-1.5, -4) {(a)};
		\node [style=none] (74) at (-1.5, -14) {};
		\node [style=none] (75) at (-1.5, -14) {$(b)$};
	\end{pgfonlayer}
	\begin{pgfonlayer}{edgelayer}
		\draw (0) to (2);
		\draw (4) to (2);
		\draw (3) to (4);
		\draw (3) to (1);
		\draw (1) to (0);
		\draw (13) to (0);
		\draw (0) to (14);
		\draw (15) to (16);
		\draw (16) to (14);
		\draw (15) to (13);
		\draw (0) to (18);
		\draw (0) to (17);
		\draw (17) to (19);
		\draw (18) to (20);
		\draw (19) to (20);
		\draw (37) to (39);
		\draw (41) to (39);
		\draw (40) to (41);
		\draw (40) to (38);
		\draw (38) to (37);
		\draw (42) to (37);
		\draw (37) to (43);
		\draw (44) to (45);
		\draw (45) to (43);
		\draw (44) to (42);
		\draw (37) to (47);
		\draw (37) to (46);
		\draw (46) to (48);
		\draw (47) to (49);
		\draw (48) to (49);
	\end{pgfonlayer}
\end{tikzpicture}
} 
 \caption{
 (a) $K$-harmonious labeling of each $C^j$ in $D_5^3$
 where $K=\Gamma/H$, $\Gamma={Z}_5\oplus {Z}_3$ and 
$H=\{0\}\oplus {Z}_3$\newline 
 (b) $\Gamma$-harmonious labeling of $D_5^3$
}
\label{D_5^3graph}
\end{figure}

%\FloatBarrier

\begin{thm}\label{thm:oddwindmill}
    When $n$ and $m$ are odd, then $D_n^m$ is $\Gamma$-harmonious for any Abelian group $\Gamma$ of order $mn$.
\end{thm}

\begin{proof}
We will use Observation~\ref{obs:blade labeling odd} to construct a vertex labeling $f$ of $D_n^m$ with $\Gamma$ that will induce a $\Gamma$-harmonious labeling of $ D_n^m$. For every $n$-cycle $C^j$, choose a distinct element $h_j$ in $H$. 
This is possible because $|H|=m$.

We start by labeling the central vertex $u$ with $\mathbf{0}$, the identity element of $\Gamma$. If a vertex $v_{i}^j$ in  some $C^j$ is labeled with $g(v^j_i)=\gamma_i+H$ in the $K$-harmonious labeling of $C^j$ constructed in Observation~\ref{obs:blade labeling odd}, then we label it with $f(v^j_i)=\gamma_i+h_j$, where $h_j$ is the chosen element of $H$ for $C^j$.

First we want to show that this construction produces a well-defined vertex labeling. Suppose that for some 
$v^{j_1}_{i_1}\neq v^{j_2}_{i_2}$ we have 
$f(v^{j_1}_{i_1})=f(v^{j_2}_{i_2})$.

$f(v^{j_1}_{i_1})=\gamma_{i_1}+h_{j_1}= \gamma_{i_2}+h_{j_2}=f(v^{j_2}_{i_2})$ and

% \textbf{Alternative way:} This implies $ \gamma_{i_1}-\gamma_{i_2}\in H$. Therefore,  by Lagrange's theorem, $ \gamma_{i_1}+H=\gamma_{i_2}+H$, but then 	$g(v^{j_1}_{i_1})=\gamma_{i_1}+H=\gamma_{i_2}+H=g(v^{j_2}_{i_2})$, and we have our contradiction because in Observation~\ref{obs:blade labeling odd} we have shown that $g$ is injective. 
%
\begin{equation}\label{eq:1}
	(\gamma_{i_1}+h_{j_1})+H=(\gamma_{i_2}+h_{j_2})+H.
\end{equation}
	Because $h_{j_1},h_{j_2}\in H$, we know that 
\begin{equation}\label{eq:2}
	(\gamma_{i_1}+h_{j_1})+H= \gamma_{i_1}+H
\end{equation}
	and
\begin{equation}\label{eq:3}
	(\gamma_{i_2}+h_{j_2})+H= \gamma_{i_2}+H.
\end{equation}
	Combining equations~\ref{eq:1},~\ref{eq:2} and~\ref{eq:3}, we obtain that
$$
	\gamma_{i_1}+H=\gamma_{i_2}+H.
$$
	But then
$$
	g(v^{j_1}_{i_1})=\gamma_{i_1}+H=\gamma_{i_2}+H=g(v^{j_2}_{i_2})
$$ 
	and we have a contradiction, because in Observation~\ref{obs:blade labeling odd} we have shown that $g$ is injective. Moreover, if for some $v^{j_0}_{i_0}$ we have $f(v^{j_0}_{i_0})=\gamma_{i_0}+h_{j_0}= \mathbf{0},$ then $ \gamma_{i_0}\in H$, which again contradicts Observation~\ref{obs:blade labeling odd}. Therefore, this construction induces an injective vertex labeling.

	Second, we must check if this construction induces a $\Gamma$-harmonious labeling of $D_n^m$. We have to check whether there are any repetitions in the edge labels $w_f(e)$ within a blade and across any two blades. 
 % In each $C^j$, we denote the vertex corresponding to the vertex $v_i$ in the original cycle $C_n$ by $v^j_i$. Then the edges in $C_j$ are $e^j_0=uv^j_1, e^j_i=v^j_i v^j_{i+1}$ for $i=1,2,\dots, n-2$, and $e^j_{n-1}=v^j_{n-1}u$.
	
	In $C^j$ when $w_g(e^j_i)=\tilde{\gamma}_i+H$, then $w_f(e^j_i)=\Tilde{\gamma}_i+kh_j$,where $k\in \{1,2\}$. In particular, for edges $uv^j_0$ and $v^j_{n-1}$  we have $k=1$, otherwise, $k=2$. 
	
	For a fixed $j$ and $0\leq i\leq n-1$, we claim that each $w_f(e^j_i)=\Tilde{\gamma}_i+kh_j$ is distinct. For $1\leq i\leq n-2$, each $ \Tilde{\gamma}_i+2h_j$ is distinct because each label $\tilde{\gamma}_i+H$ in $C^j$ was distinct. Moreover, $\Tilde{\gamma}_0+h_j\neq \Tilde{\gamma}_{n-1}+h_j$ for the same reason.
	
	Now suppose that for a fixed $j$ and some $i$,  
	$1\leq i\leq n-2$, we have 
	$\Tilde{\gamma}_0+h_j=\Tilde{\gamma}_{i}+2h_j$. Then 
	$\Tilde{\gamma}_{0}=\Tilde{\gamma}_{i}+h_j$
	and
$$
	w_g(e_0)=\Tilde{\gamma}_{0}+H=(\Tilde{\gamma}_{i}+h_j)+H
		=\Tilde{\gamma}_{i}+H =w_g(e_i),
$$	
	which is a contradiction by Observation~\ref{obs:blade labeling odd}. Similarly for 
	$ 2\leq i\leq n-2, w_f(e^j_{n-1})=\Tilde{\gamma}_{n-1}+h_j\neq 	\Tilde{\gamma}_{i}+2h_j=w_f(e^j_i)$.

%\vskip10pt
% \marbl{Aren't we repeating here the arguments above? Notice I have not changed the subscripts here. They are incorrect. In your notation they should be 1 and $n$, in the new notation $0$ and $n-1$.
% \newline
% \textbf{Please check that out and let me know.}
% % }
% \bblue{Additionally, for 
% $ \Tilde{\gamma}_1$ and $\Tilde{\gamma}_{n-1}$, each
% $ \Tilde{\gamma}_i+h_j$ and each $ \Tilde{\gamma}_{n-1}+h_j$ is distinct for each $ h_j$. Moreover,  when $ 2\leq i\leq n,\Tilde{\gamma}_{i} +2h_j$ is also distinct for each $h_j$ because $H=2H$.}

%\vskip30pt
	Finally suppose for $0\leq j_1<j_2\leq m$, and $0\leq i_1\leq i_2\leq n-1$ two edge labels $w_f(e^{j_1}_{i_1})$ and $w_f(e^{j_2}_{i_2})$ across  cycles $C^{j_1}$ and $C^{j_2}$ are the same. 
 
 % Here $ j_1$ and $j_2$ may be equal.
 
	Then
	$\Tilde{\gamma}_{i_1}+k_1 h_{j_1}= \Tilde{\gamma}_{i_2}+k_2 h_{j_2}$, where $k_1,k_2\in \{1,2\}$. Then similarly as above we have
$$
	\Tilde{\gamma}_{i_1}=\Tilde{\gamma}_{i_2}+k_2h_{j_2}-k_1 h_{j_1}
		=\Tilde{\gamma}_{i_2}+h^*
$$ 
% \margr{Now when I look at this, maybe we can do both cases at once, $j_1=j_2$ and $j_1\neq j_2$, because the proof is almost identical.
% \newline
% \textbf{I no longer think it would be a good idea. If you agree, please remove this comment in your next draft, or tell me an I remove it if it is me who writes the next draft first.}
% }	
	for some $h^*\in H$. It follows that
$$
	w_g(e_{i_1})=\Tilde{\gamma}_{i_1}+H=(\Tilde{\gamma}_{i_2}+h^*)+H
		=\Tilde{\gamma}_{i^2}+H =w_g(e_{i_2}),
$$	
	a contradiction again.	Thus, each edge label is distinct and the proof is complete. 
\end{proof}
When $n$ is even, we have a weaker theorem.

\begin{thm}
\label{thm:evenwindmill}
For $n$ even and $m$ odd, $D_n^m$ is $\Gamma$-harmonious when $|\Gamma|=mn$, and
$\Gamma$ has a  subgroup $H$ or order $m$ such that $K=\Gamma/H$ satisfies the conditions in Theorem \ref{thm:all-cycle}.
\end{thm}
\begin{proof}
    Let $m$ be odd and $n$ be even.
    If $\Gamma$ has a subgroup of order $m$ such that $K=\Gamma/H$ satisfies the conditions in Theorem \ref{thm:all-cycle}.
    Then, each $C^j$ is $K$-harmonious labeling where $u$ is labeled with $H$. 
   The construction then is same as in Theorem \ref{thm:oddwindmill}.
\end{proof}

When $m$ is even, this construction does not work.
As discussed in Theorem \ref{thm:oddwindmill}, the label of edge $e_i^j$ is $\Tilde{\gamma}_i+kh_j$ for $k=\{1,2\}$ and $h_j\in H$. We know $|H|=m$. Thus in $H$, we can have two distinct elements $h_j$ and $h_{j'}$ such that $2h_j=2h_{j'}$.

%\newpage
\section{Generalized prisms}\label{sec:prisms}

For generalized prisms, our result is again more general for $n$ odd. That is due to the fact that even cycles are not $\Gamma$-harmonious for all $\Gamma$ of the proper order, while for $n$ odd is it $\Gamma$-harmonious for all Abelian groups $\Gamma$ of order $n$ (see Theorem~\ref{thm:harm-groups} or~\ref{thm:all-cycle}).

Recall that the number of edges of the generalized prism $Y_{m,n}$ is $(2m-1)n$. The main idea of our construction is the following. Given a group $\Gamma$ of order $(2m-1)n$ and a subgroup $H$ of $\Gamma$ of order $n$ satisfying assumptions of Theorem~\ref{thm:all-cycle} and assume that there is an $H$-harmonious labeling of $C_n$. Then we label the vertices of each $C_n$ by elements of a coset $\beta+H$ where $\beta$ is a generator of the cyclic quotient group $\Gamma/H$ in a way corresponding to the original $H$-harmonious labeling. This way the edge labels of every ``layer'', that is, either edges of a particular cycle or all edges between two consecutive cycles form a coset of $\Gamma$.

We illustrate the method in two examples.

%\newpage
%%%%%%%%%%%%%%%%%%%%%%%%%%%%%%%%%%%%%%%
%%%%%%%%%%%%%%%%%%%%%%%%%%%%%%%%%%%%%%%
%%%%%%%%%%%%%%%%%%%%%%%%%%%%%%%%%%%%%%%

\begin{exm}\label{exm:C_5 and Y_3,5}
	We present a labeling of $Y_{3,5}$ with $Z_5\oplus Z_5$, starting with a labeling of $C_5$ with $Z_5$.	
\end{exm}

%%%%%%%%%%%%%%%%%%%%%%%%%%%%%%%%%%%%%%%

\begin{figure}[h!]\label{fig: C_5 with Z_5}
	\centering
	
	\scalebox{0.55}{\begin{tikzpicture}
    % Define the radius for the inner and layertwo pentagons
    \def\Radiusone{3.5}
    \def\Radiusoneinv{4.0}
 \def\labelsone{{"0", "1", "2", "3", "4"}}

    % Draw the inner pentagon
    \foreach \i in {0,1,2,3,4} {
        % Inner pentagon vertices
        \node[draw, circle, inner sep=2pt] (layerone\i) at (72*\i:\Radiusone cm) {};
    }
    \draw (layerone0) -- (layerone1) -- (layerone2) -- (layerone3) -- (layerone4) -- (layerone0);

\foreach \i in {0,1,2,3,4} {
        % Middle pentagon invisible vertices with custom labels
        \node at (72*\i:\Radiusoneinv cm) {\pgfmathparse{\labelsone[\i]}\pgfmathresult};
    }  
\end{tikzpicture}
}
	
	\caption{Labeling of $C_5$ with $Z_5$}
\end{figure}
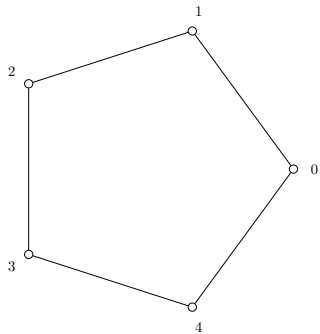
%%%%%%%%%%%%%%%%%%%%%%%%%%%%%%%%%%%%%%%
\begin{figure}[h!]\label{fig:C_5 with H}
	\centering
	
	\scalebox{0.55}{\begin{tikzpicture}
    % Define the radius for the inner and layertwo pentagons
    \def\Radiusone{3.5}
    \def\Radiusoneinv{4.2}
 \def\labelsone{{"$(0,0)$", "$(1,0)$", "$(2,0)$", "$(3,0)$", "$(4,0)$"}}
    % Draw the inner pentagon
    \foreach \i in {0,1,2,3,4} {
        % Inner pentagon vertices
        \node[draw, circle, inner sep=2pt] (layerone\i) at (72*\i:\Radiusone cm) {};
    }
    \draw (layerone0) -- (layerone1) -- (layerone2) -- (layerone3) -- (layerone4) -- (layerone0);

\foreach \i in {0,1,2,3,4} {
        % Middle pentagon invisible vertices with custom labels
        \node at (72*\i:\Radiusoneinv cm) {\pgfmathparse{\labelsone[\i]}\pgfmathresult};
    }  
\end{tikzpicture}
} 
	\caption{Labeling of $C_{5} $ with $Z_5\oplus \{0\}$}
\end{figure}
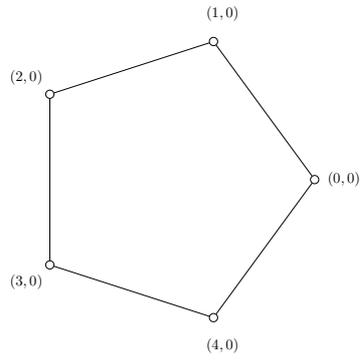
%%%%%%%%%%%%%%%%%%%%%%%%%%%%%%%%%%%%%%%
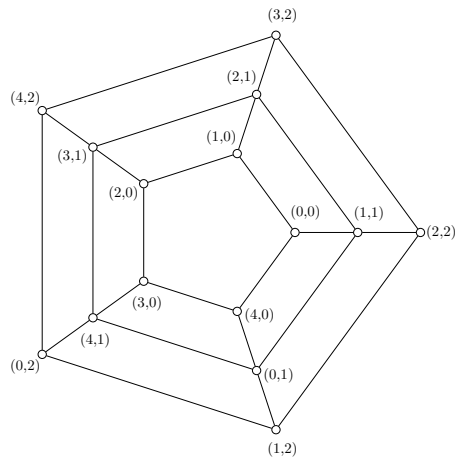
\begin{figure}[h!]\label{fig:genprism Y_{3,5}}
	\centering
	
	\scalebox{0.55}{
 \begin{tikzpicture}
    % Define the radius for the inner and layertwo pentagons
    \def\Radiusone{2.0}
    \def\Radiusoneinv{2.3}
    \def\Radiustwo{3.5}
    \def\Radiustwoinv{3.8}
    \def\Radiusthree{5}
    \def\Radiusthreeinv{5.5}  
 \def\labelsone{{"(0,0)", "(1,0)", "(2,0)", "(3,0)", "(4,0)"}}
\def\labelstwo{{"(1,1)", "(2,1)", "(3,1)", "(4,1)", "(0,1)"}}
\def\labelsthree{{"(2,2)", "(3,2)", "(4,2)", "(0,2)", "(1,2)"}}   

    % Draw the inner pentagon
    \foreach \i in {0,1,2,3,4} {
        % Inner pentagon vertices
        \node[draw, circle, inner sep=2pt] (layerone\i) at (72*\i:\Radiusone cm) {};
    }
    \draw (layerone0) -- (layerone1) -- (layerone2) -- (layerone3) -- (layerone4) -- (layerone0);

\foreach \i in {0,1,2,3,4} {
        % Middle pentagon invisible vertices with custom labels
        \node at (72*\i+12:\Radiusoneinv cm) {\pgfmathparse{\labelsone[\i]}\pgfmathresult};
    }

    % Draw the layertwo pentagon
    \foreach \i in {0,1,2,3,4} {
        % Outer pentagon vertices
        \node[draw, circle, inner sep=2pt] (layertwo\i) at (72*\i:\Radiustwo cm) {};
    }
    \draw (layertwo0) -- (layertwo1) -- (layertwo2) -- (layertwo3) -- (layertwo4) -- (layertwo0);
\foreach \i in {0,1,2,3,4} {
        % Middle pentagon invisible vertices with custom labels
        \node at (72*\i+7:\Radiustwoinv cm) {\pgfmathparse{\labelstwo[\i]}\pgfmathresult};
    }

     % Draw the inner pentagon
    \foreach \i in {0,1,2,3,4} {
        % Inner pentagon vertices
        \node[draw, circle, inner sep=2pt] (layerthree\i) at (72*\i:\Radiusthree cm) {};
    }
    \draw (layerthree0) -- (layerthree1) -- (layerthree2) -- (layerthree3) -- (layerthree4) -- (layerthree0);
\foreach \i in {0,1,2,3,4} {
        % Middle pentagon invisible vertices with custom labels
        \node at (72*\i:\Radiusthreeinv cm) {\pgfmathparse{\labelsthree[\i]}\pgfmathresult};
    }

    \foreach \i in {0,1,2,3,4} {
        \draw (layerone\i) -- (layertwo\i);
         \draw (layertwo\i) -- (layerthree\i);
        
    }
    
\end{tikzpicture}
}
	
	\caption{Labeling of $Y_{3,5}$ with $Z_5\oplus Z_5$}
\end{figure}
%%%%%%%%%%%%%%%%%%%%%%%%%%%%%%%%%%%%%%%
%%%%%%%%%%%%%%%%%%%%%%%%%%%%%%%%%%%%%%%
%%%%%%%%%%%%%%%%%%%%%%%%%%%%%%%%%%%%%%%

\newpage
%%%%%%%%%%%%%%%%%%%%%%%%%%%%%%%%%%%%%%%
%%%%%%%%%%%%%%%%%%%%%%%%%%%%%%%%%%%%%%%
%%%%%%%%%%%%%%%%%%%%%%%%%%%%%%%%%%%%%%%

\begin{exm}\label{exm:C_8 and Y_3,8}
	We present a labeling of $Y_{3,8}$ with $Z_4\oplus Z_2\oplus Z_5$, starting with a labeling of $C_8$ with $Z_4\oplus Z_2$.	
\end{exm}

%%%%%%%%%%%%%%%%%%%%%%%%%%%%%%%%%%%%%%%

\begin{figure}[h!]\label{fig: C_8 with Z_8}
	\centering
	
	\scalebox{0.55}{
\begin{tikzpicture}
    % Define the radius for the inner and layertwo octagons
 \def\Radiusone{3.5}
    \def\Radiusoneinv{4.2}

    % Define labels for each layer (if needed)
      \def\labelsone{{"(0,0)", "(2,0)", "(1,0)", "(2,1)", "(3,0)", "(3,1)", "(1,1)", "(0,1)"}}

    % Draw the inner octagon
    \foreach \i in {0,...,7} {
        % Inner octagon vertices
        \node[draw, circle, inner sep=2pt] (layerone\i) at (45*\i:\Radiusone cm) {};
    }
    \draw (layerone0) -- (layerone1) -- (layerone2) -- (layerone3) -- (layerone4) -- (layerone5) -- (layerone6) -- (layerone7) -- (layerone0);

    % Invisible vertices with custom labels for the inner octagon
    \foreach \i in {0,...,7} {
        \node at (45*\i:\Radiusoneinv cm) {\pgfmathparse{\labelsone[\i]}\pgfmathresult};
    }

\end{tikzpicture}
}
	
	\caption{Labeling of $C_8$ with $Z_4\oplus Z_2$}
\end{figure}

%%%%%%%%%%%%%%%%%%%%%%%%%%%%%%%%%%%%%%
\begin{figure}[h!]\label{fig:genprism Y_{3,8}}
	\centering
	
	\scalebox{0.5}{
\begin{tikzpicture}
    % Define the radius for the inner and layertwo octagons
    \def\Radiusone{2.0}
    \def\Radiusoneinv{2.4}
    \def\Radiustwo{3.5}
    \def\Radiustwoinv{3.9}
    \def\Radiusthree{5}
    \def\Radiusthreeinv{5.7}

    % Define labels for each layer (if needed)
    \def\labelsone{{"(0,0,0)", "(2,0,0)", "(1,0,0)", "(2,1,0)", "(3,0,0)", "(3,1,0)", "(1,1,0)", "(0,1,0)"}}
    
    \def\labelstwo{{ "(2,0,1)", "(1,0,1)", "(2,1,1)", "(3,0,1)", "(3,1,1)", "(1,1,1)", "(0,1,1)","(0,0,1)"}}
    \def\labelsthree{{ "(1,0,2)", "(2,1,2)", "(3,0,2)", "(3,1,2)", "(1,1,2)", "(0,1,2)","(0,0,2)", "(2,0,2)"}}

    % Draw the inner octagon
    \foreach \i in {0,...,7} {
        % Inner octagon vertices
        \node[draw, circle, inner sep=2pt] (layerone\i) at (45*\i:\Radiusone cm) {};
    }
    \draw (layerone0) -- (layerone1) -- (layerone2) -- (layerone3) -- (layerone4) -- (layerone5) -- (layerone6) -- (layerone7) -- (layerone0);

    % Invisible vertices with custom labels for the inner octagon
    \foreach \i in {0,...,7} {
        \node at (45*\i+15:\Radiusoneinv cm) {\pgfmathparse{\labelsone[\i]}\pgfmathresult};
    }

    % Draw the layertwo octagon
    \foreach \i in {0,...,7} {
        % layertwo octagon vertices
        \node[draw, circle, inner sep=2pt] (layertwo\i) at (45*\i:\Radiustwo cm) {};
    }
    \draw (layertwo0) -- (layertwo1) -- (layertwo2) -- (layertwo3) -- (layertwo4) -- (layertwo5) -- (layertwo6) -- (layertwo7) -- (layertwo0);

    % Invisible vertices with custom labels for the layertwo octagon
    \foreach \i in {0,...,7} {
        \node at (45*\i+10:\Radiustwoinv cm) {\pgfmathparse{\labelstwo[\i]}\pgfmathresult};
    }

    % Draw the layerthree octagon
    \foreach \i in {0,...,7} {
        % layerthree octagon vertices
        \node[draw, circle, inner sep=2pt] (layerthree\i) at (45*\i:\Radiusthree cm) {};
    }
    \draw (layerthree0) -- (layerthree1) -- (layerthree2) -- (layerthree3) -- (layerthree4) -- (layerthree5) -- (layerthree6) -- (layerthree7) -- (layerthree0);

    % Invisible vertices with custom labels for the layerthree octagon
    \foreach \i in {0,...,7} {
        \node at (45*\i:\Radiusthreeinv cm) {\pgfmathparse{\labelsthree[\i]}\pgfmathresult};
    }

    % Connect vertices between layers
    \foreach \i in {0,...,7} {
        \draw (layerone\i) -- (layertwo\i);
        \draw (layertwo\i) -- (layerthree\i);
    }
\end{tikzpicture}
} 
	\caption{Labeling of $Y_{3,8} $ with $Z_4\oplus Z_2\oplus Z_5$}
	
\end{figure}
%%%%%%%%%%%%%%%%%%%%%%%%%%%%%%%%%%%%%%%
%%%%%%%%%%%%%%%%%%%%%%%%%%%%%%%%%%%%%%%
%%%%%%%%%%%%%%%%%%%%%%%%%%%%%%%%%%%%%%%

\newpage

\begin{thm}\label{thm:all_prisms}
	Let $n\geq3,m\geq2$ and $\Gamma$ be an Abelian group of order $n(2m-1)$. If $\Gamma$ has a subgroup $H$ of order $n$ that satisfies the assumptions of Theorem~\ref{thm:all-cycle}, and the quotient group $\Gamma/H$ is cyclic, then prism $Y_{m,n}\cong P_m\Box C_n$ admits a $\Gamma$-harmonious labeling.
\end{thm}

\begin{proof}
	We will use the following notation. Each $n$-cycle $C^j$ for $j=0,1,\dots,m-1$ will consist of vertices $x^j_i$ with $i=1,2,\dots,n$ and edges $x^j_ix^j_{i+1}$, where $i=1,2,\dots, n$ and addition in the subscript is performed modulo $n$. The paths $P^i$ will consist of vertices $x^0_i, x^1_{i+1},\dots,x^{m-1}_{i+m-1}$ and edges 
	$x^j_i x^{j+1}_{i+1}$ where $i=1,2,\dots, n, j=0,1,\dots,m-2$ and addition in the subscript is again performed modulo $n$. Thus, the path edges can be seen as ``diagonal'', which is done for computational convenience.

	Let ${H}$ be a subgroup $\Gamma$ of order $n$ that satisfies the assumptions of Theorem~\ref{thm:all-cycle} and $K=\Gamma/{H}$ be cyclic of order $2m-1$ generated by a coset $\beta+H$.

	Hence there exists an $H$-harmonious labeling $g$ of $C_n=x_1,x_2,\dots,x_n$. 
	For each $x^j_i$ we now define
	$$
		f(x^j_i)=g(x_i) +j\beta.
	$$
	Then the label $w(x^j_i x^j_{i+1})$ of an edge $x^j_i x^j_{i+1}$ in cycle $C^j$ is defined as
	\begin{align*}
		w(x^j_i x^j_{i+1})	&= f(x^j_i) +f(x^j_{i+1})\\
							&= (g(x_i) +j\beta) +(g(x_{i+1}) +j\beta)\\ 
							&= g(x_i) +g(x_{i+1}) +2j\beta \\
							&= w(x_i x_{i+1}) +2j\beta.
	\end{align*}
	Now the set of all edge labels in $C^j$ is
	\begin{align*}
		W^j	&=\{w(x^j_i x^j_{i+1})\mid i=1,2,\dots,n	\} \\
			&=\{w(x_i x_{i+1}) +2j\beta \mid i=1,2,\dots,n	\}.
	\end{align*}

	Because $g$ is an $H$-harmonious labeling, the set of all weights $w(x_i x_{i+1})$ forms the subgroup $H$ itself and 
	\begin{align*}
		W^j	&=\{w(x_i x_{i+1}) +2j\beta \mid i=1,2,\dots,n	\} =2j\beta+H,
	\end{align*}
	an even coset in  $\Gamma/H$. Consequently, the labels of all edges in cycles $C^0,C^1,\dots, C^m$ form all $m$ even cosets 
	$H, 2\beta +H,\dots, 2(m-1)\beta +H$ of the cyclic group $\Gamma/H$.

	For paths, the label of $x^j_i x^{j+1}_{i+1}$ is defined as
	\begin{align*}
		w(x^j_i x^{j+1}_{i+1})	&= f(x^j_i) +f(x^{j+1}_{i+1})\\ 
							&= (g(x_i) +j\beta) +(g(x_{i+1}) +(j+1)\beta) \\
							&= g(x_i) +g(x_{i+1}) +(2j+1)\beta \\
							&= w(x_i x_{i+1}) +(2j+1)\beta.	
	\end{align*}
	The set of labels of the path edges between cycles $C^j$ and $C^{j+1}$ 	for $j=0,1,\dots,m-2$ is then
	\begin{align*}
		U^j	&=\{w(x^j_i x^{j+1}_{i+1})\mid i=1,2,\dots,n	\} \\
			&=\{w(x_i x_{i+1}) +(2j+1)\beta \mid i=1,2,\dots,n	\}
	\end{align*}
	the coset $(2j+1)\beta+H$. Therefore, the weights of all path edges form all $m-1$ odd cosets $\beta+H, 3\beta +H,\dots, (2m-3)\beta +H$ of $\Gamma/H$.
	
	Because we have $j=0,1,\dots,m-1$, we obtained $2m-1$ distinct cosets and their union forms the group $\Gamma$. This completes the proof.
\end{proof} 

For convenience, we express the above result in more explicit form, describing the group $\Gamma$ in the nested form.

\begin{cor}\label{cor:odd-prisms}
	Let $n\geq3$ be odd, $m\geq2$ and $\Gamma$ an Abelian group of order $(2m-1)n$ with a cyclic subgroup $K$ of order $2m-1$. Then the prism $Y_{m,n}\cong C_n\Box P_m$ admits a $\Gamma$-harmonious labeling.    
\end{cor}

\begin{proof}
    By Theorem~\ref{thm:fact-group-isom-to-subgroup} part 2, there exists a subgroup $H$ such that $K\approx \Gamma/H$. Indeed, $H$ is of order $n$ and $n$ is odd. Thus, by Theorem~\ref{thm:all-cycle} there exists a $H$-harmonious labeling of $C_n$. The result now follows from Theorem~\ref{thm:all_prisms}.
\end{proof}

\newpage
\begin{cor}\label{cor:even-prisms}
	Let $n\geq4$ be even, $m\geq2$ and $\Gamma$ an Abelian group {of order $(2m-1)n$}  
	written in the nested form as $\Gamma=Z_{n_1}\oplus Z_{n_2}\oplus\dots\oplus Z_{n_t}$ where $t\geq2$, with a cyclic subgroup $K$ of order $2m-1$. If
	\begin{itemize}
	\item[(1)] 	$n_1\equiv0\pmod4$ and $n_2$ is even, or 
	\item[(2)]	$n_2\geq4$ is even,
	\end{itemize}
	then the prism $Y_{m,n}\cong C_n\Box P_m$ admits a $\Gamma$-harmonious labeling.    
\end{cor}

\begin{proof}
	{By Observation~\ref{obs:cyclic subgroup}, $2m-1$ divides $n_1$. By Theorem~\ref{thm:fact-group-isom-to-subgroup} part 2, there is a subgroup $H$ such that $K\approx \Gamma/H$. Moreover, $H$ can be written as
	$H=Z_{n'_1}\oplus Z_{n_2}\oplus\dots\oplus Z_{n_t}$, where $n'_1=n_1/(2m-1)$.   
	
	When $n_1\equiv0\pmod4$ and $n_2$ is even, then also 
	$n'_1\equiv0\pmod4$ and $H$ contains a subgroup isomorphic to $Z_4\oplus Z_2$ and satisfies assumptions of Theorem~\ref{thm:harm-groups}. The result then follows from Theorems~\ref{thm:all-cycle} and ~\ref{thm:all_prisms}.}

	When $n_2\geq4$ is even, then $n_1$ is even as well. And because $n_2\geq4$, $H$ contains a subgroup isomorphic to $Z_2\oplus Z_{2s}$ for $s\geq2$. Therefore, the result follows again from Theorems~\ref{thm:harm-groups},~\ref{thm:all-cycle} and ~\ref{thm:all_prisms}.
\end{proof}

%\newpage
\section{Generalized web graphs}\label{sec:webs}

For computation convenience, we slightly change the notation here and number the cycles of the generalized prism $Y_{m,n}$ as $C^1,C^2,\dots, C^m$.

A \emph{generalized closed web} is then a graph arising from  $Y_{m,n}$ by adding a vertex $x^0_{0}$ and joining it by an edge to every vertex of the top cycle $C^1$ of $Y_{m,n}$. The remaining vertices will follow the notation from Section~\ref{sec:prisms}.

The number of edges and order of $\Gamma$ is now $2mn$.

\begin{thm}\label{thm:closed_web}
	Let $n\geq3,m\geq2$ and $\Gamma$ be an Abelian group of order $2mn$. If $\Gamma$ has a subgroup $H$ of order $n$ that satisfies assumptions of Theorem~\ref{thm:all-cycle} and the quotient group $\Gamma/H$ is cyclic, 
	then the generalized closed web $CW_{m,n}$ admits a $\Gamma$-harmonious labeling. 
\end{thm}

\begin{proof}
	
	We will use the following notation. The central vertex of degree $n$ will be denoted by $x^0_{0}$. Each $n$-cycle $C^j$ for $j=1,2,\dots,m$ will consist of vertices $x^j_i$ with $i=1,2,\dots,n$ and edges $x^j_ix^j_{i+1}$, where $i=1,2,\dots, n$ and addition in the subscript is performed modulo $n$. The paths $P^i$ will consist of vertices $x^0_0, x^1_{i},x^2_{i+1}\dots,x^{m-1}_{i+m-2}$ and edges 
	$x^0_0x^1_{i}$ and $x^j_{i+j-1} x^{j+1}_{i+j}$ where $i=1,2,\dots, n, \ j=1,2,\dots,m-1$ and addition in the subscript is again performed modulo $n$. 
	
	The number of edges in $CW_{m,n}$ is  $2mn$, which implies $|\Gamma|=2mn$. 

	By our assumption, $H$ satisfies conditions of Theorem~\ref{thm:all-cycle} and therefore there exists an $H$-harmonious labeling $g$ of $C_n=x_1,x_2,\dots,x_n$. 
	We first label the central vertex by the identity element of $\Gamma$, that is,
	$$
		f(x^0_0) = e.
	$$
	The remaining vertices will be labeled as
	$$
		f(x^j_i)=g(x_i) +j\beta.
	$$
	Then the label $w(x^j_i x^j_{i+1})$ of an edge $x^j_i x^j_{i+1}$ in cycle $C^j$ is defined as
	\begin{align*}
		w(x^j_i x^j_{i+1})	&= f(x^j_i) +f(x^j_{i+1})\\
							&= g(x_i) +g(x_{i+1}) +2j\beta \\
							&= w(x_i x_{i+1}) +2j\beta	
	\end{align*}
	and the set of all edge labels in $C^j$ is
	\begin{align*}
		W^j	&=\{w(x^j_i x^j_{i+1})\mid i=1,2,\dots,n	\} \\
			&=\{w(x_i x_{i+1}) +2j\beta \mid i=1,2,\dots,n	\} \\
			&= 2j\beta +H,
	\end{align*}
	an even coset in the quotient group $\Gamma/H$, because the set of weights of the edges in the cycle $C_n$ forms the subgroup $H$. This follows from our assumption that $H$ allows an $H$-harmonious labeling $g$. 
	
	Notice that in the cycle $C^m$ we have
	\begin{align*}
		W^m	&= 2m\beta +H = H,
	\end{align*}
	because $\Gamma/H$ is cyclic of order $2m$. Therefore, the labels of all edges in cycles $C^1,C^2,\dots, C^m$ form all $m$ even cosets 
	$H, 2\beta +H,\dots, (2m-2)\beta +H$ of the cyclic group $\Gamma/H$.

	For paths, the label of $x^j_i x^{j+1}_{i+1}$ is defined as
	\begin{align*}
		w(x^j_i x^{j+1}_{i+1})	&= f(x^j_i) +f(x^{j+1}_{i+1})\\ 
							&= g(x_i) +g(x_{i+1}) +(2j+1)\beta \\
							&= w(x_i x_{i+1}) +(2j+1)\beta.
	\end{align*}
	The set of labels of the edges incident with the central vertex is
	\begin{align*}
		U^0	&=\{w(x^0_0 x^{1}_{i})\mid i=1,2,\dots,n	\} \\
			&=\{f(x^0_0) + f(x^1_i) \mid i=1,2,\dots,n	\} \\
			&=\{e + g(x_i) +\beta \mid i=1,2,\dots,n	\} \\
			&= \beta +H,
	\end{align*}
	because $g$ is the $H$-harmonious labeling of the cycle $C_n$ and therefore a bijection from the vertex set of $C_n$ to $H$.

	For the remaining path edges, the set of labels between cycles $C^j$ and $C^{j+1}$ 	for $j=1,2,\dots,m-1$ is 
	\begin{align*}
		U^j	&=\{w(x^j_i x^{j+1}_{i+1})\mid i=1,2,\dots,n	\} \\
			&=\{w(x_i x_{i+1}) +(2j+1)\beta \mid i=1,2,\dots,n	\} \\
			&= (2j+1) +H.
	\end{align*}
	Hence, the weights of all path edges form all $m$ odd cosets $\beta+H, 3\beta +H,\dots, (2m-1)\beta +H$ of $\Gamma/H$.
	
	We obtained $2m$ distinct cosets and their union forms the group $\Gamma$. This completes the proof.
\end{proof}

%%%%%%%%%%%%%%%%%%%%%%%%%%%%
%%%%%%%%%%%%%%%%%%%%%%%%%%%%
%%%%%%%%%%%%%%%%%%%%%%%%%%%%
%%%%%%%%%%%%%%%%%%%%%%%%%%%%
%\newpage
\begin{exm}\label{exm:CW_3,5}
	A $Z_5\oplus Z_6$-harmonious labeling of the closed web $CW_{3,5}$.
\end{exm}

\begin{figure}[h!]\label{fig:CW_{3,5}}
	\centering
	
	\scalebox{0.55}{\begin{tikzpicture}
    % Define the radius for the inner and layertwo pentagons
      \def\Radiusone{2.0}
    \def\Radiusoneinv{2.3}
    \def\Radiustwo{3.5}
    \def\Radiustwoinv{3.8}
    \def\Radiusthree{5}
    \def\Radiusthreeinv{5.5}  
 \def\labelsone{{"(0,1)", "(1,1)", "(2,1)", "(3,1)", "(4,1)"}}
\def\labelstwo{{"(1,2)", "(2,2)", "(3,2)", "(4,2)","(0,2)" }}
\def\labelsthree{{"(2,3)", "(3,3)", "(4,3)", "(0,3)", "(1,3)"}}

    % Draw the inner pentagon
    \foreach \i in {0,1,2,3,4} {
        % Inner pentagon vertices
        \node[draw, circle, inner sep=2pt] (layerone\i) at (72*\i:\Radiusone cm) {};
    }
    \draw (layerone0) -- (layerone1) -- (layerone2) -- (layerone3) -- (layerone4) -- (layerone0);

\foreach \i in {0,1,2,3,4} {
        % Middle pentagon invisible vertices with custom labels
        \node at (72*\i+12:\Radiusoneinv cm) {\pgfmathparse{\labelsone[\i]}\pgfmathresult};
    }

    % Draw the layertwo pentagon
    \foreach \i in {0,1,2,3,4} {
        % Outer pentagon vertices
        \node[draw, circle, inner sep=2pt] (layertwo\i) at (72*\i:\Radiustwo cm) {};
    }
    \draw (layertwo0) -- (layertwo1) -- (layertwo2) -- (layertwo3) -- (layertwo4) -- (layertwo0);
\foreach \i in {0,1,2,3,4} {
        % Middle pentagon invisible vertices with custom labels
        \node at (72*\i+7:\Radiustwoinv cm) {\pgfmathparse{\labelstwo[\i]}\pgfmathresult};
    }

     % Draw the inner pentagon
    \foreach \i in {0,1,2,3,4} {
        % Inner pentagon vertices
        \node[draw, circle, inner sep=2pt] (layerthree\i) at (72*\i:\Radiusthree cm) {};
    }
    \draw (layerthree0) -- (layerthree1) -- (layerthree2) -- (layerthree3) -- (layerthree4) -- (layerthree0);
\foreach \i in {0,1,2,3,4} {
        % Middle pentagon invisible vertices with custom labels
        \node at (72*\i:\Radiusthreeinv cm) {\pgfmathparse{\labelsthree[\i]}\pgfmathresult};
    }

    \foreach \i in {0,1,2,3,4} {
        \draw (layerone\i) -- (layertwo\i);
         \draw (layertwo\i) -- (layerthree\i);
        
    }
    \node[draw, circle, inner sep=2pt] (center) at (0,0) {};
    \node (centerlabel) at (0.6,0.4) {(0,0)};
 % Connect the central node to the cycle nodes
    \foreach \i in {0,1,2,3,4} {
        \draw (center) -- (layerone\i);
    }
    
\end{tikzpicture}

} 
	\caption{Labeling of $CW_{3,5} $ with $Z_5\oplus  Z_6$}
	
\end{figure}
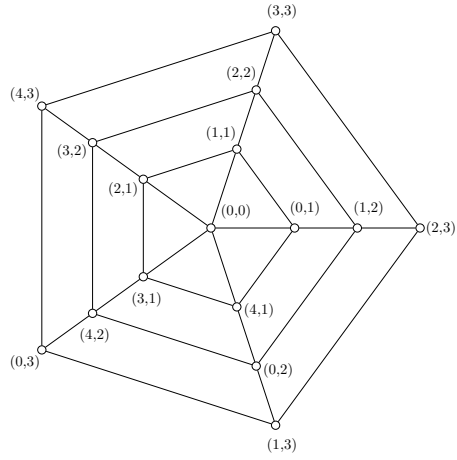
%%%%%%%%%%%%%%%%%%%%%%%%%%%%
\begin{exm}\label{exm:CW_3,8}
	A $Z_4\oplus Z_2\oplus Z_6$-harmonious labeling of the closed web $CW_{3,8}$.
\end{exm}

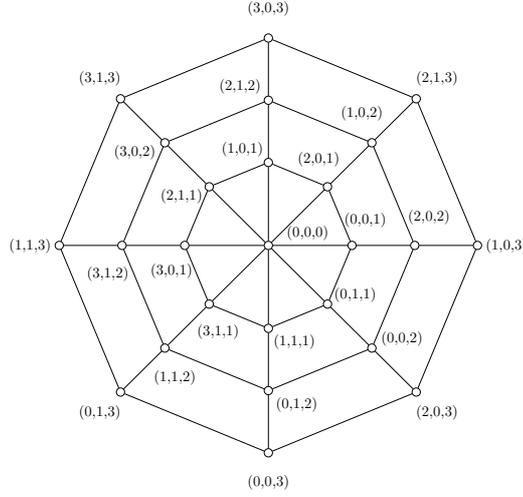
\begin{figure}[h!]\label{fig:CW_{3,8}}
	\centering
	
	\scalebox{0.55}{
\begin{tikzpicture}
    % Define the radius for the inner and layertwo octagons
    \def\Radiusone{2.0}
    \def\Radiusoneinv{2.4}
    \def\Radiustwo{3.5}
    \def\Radiustwoinv{3.9}
    \def\Radiusthree{5}
    \def\Radiusthreeinv{5.7}

    % Define labels for each layer (if needed)
    \def\labelsone{{"(0,0,1)", "(2,0,1)", "(1,0,1)", "(2,1,1)", "(3,0,1)", "(3,1,1)", "(1,1,1)", "(0,1,1)"}}
    
    \def\labelstwo{{"(2,0,2)", "(1,0,2)", "(2,1,2)", "(3,0,2)", "(3,1,2)", "(1,1,2)", "(0,1,2)","(0,0,2)" }}
    \def\labelsthree{{ "(1,0,3)", "(2,1,3)", "(3,0,3)", "(3,1,3)", "(1,1,3)", "(0,1,3)","(0,0,3)", "(2,0,3)"}}

    % Draw the inner octagon
    \foreach \i in {0,...,7} {
        % Inner octagon vertices
        \node[draw, circle, inner sep=2pt] (layerone\i) at (45*\i:\Radiusone cm) {};
    }
    \draw (layerone0) -- (layerone1) -- (layerone2) -- (layerone3) -- (layerone4) -- (layerone5) -- (layerone6) -- (layerone7) -- (layerone0);

    % Invisible vertices with custom labels for the inner octagon
    \foreach \i in {0,...,7} {
        \node at (45*\i+15:\Radiusoneinv cm) {\pgfmathparse{\labelsone[\i]}\pgfmathresult};
    }

    % Draw the layertwo octagon
    \foreach \i in {0,...,7} {
        % layertwo octagon vertices
        \node[draw, circle, inner sep=2pt] (layertwo\i) at (45*\i:\Radiustwo cm) {};
    }
    \draw (layertwo0) -- (layertwo1) -- (layertwo2) -- (layertwo3) -- (layertwo4) -- (layertwo5) -- (layertwo6) -- (layertwo7) -- (layertwo0);

    % Invisible vertices with custom labels for the layertwo octagon
    \foreach \i in {0,...,7} {
        \node at (45*\i+10:\Radiustwoinv cm) {\pgfmathparse{\labelstwo[\i]}\pgfmathresult};
    }

    % Draw the layerthree octagon
    \foreach \i in {0,...,7} {
        % layerthree octagon vertices
        \node[draw, circle, inner sep=2pt] (layerthree\i) at (45*\i:\Radiusthree cm) {};
    }
    \draw (layerthree0) -- (layerthree1) -- (layerthree2) -- (layerthree3) -- (layerthree4) -- (layerthree5) -- (layerthree6) -- (layerthree7) -- (layerthree0);

    % Invisible vertices with custom labels for the layerthree octagon
    \foreach \i in {0,...,7} {
        \node at (45*\i:\Radiusthreeinv cm) {\pgfmathparse{\labelsthree[\i]}\pgfmathresult};
    }

    % Connect vertices between layers
    \foreach \i in {0,...,7} {
        \draw (layerone\i) -- (layertwo\i);
        \draw (layertwo\i) -- (layerthree\i);
    }

    \node[draw, circle, inner sep=2pt] (center) at (0,0) {};
    \node (centerlabel) at (0.94,0.3) {(0,0,0)};
 % Connect the central node to the cycle nodes
    \foreach \i in {0,1,2,3,4,5,6,7} {
        \draw (center) -- (layerone\i);
    }
\end{tikzpicture}
} 
	\caption{Labeling of $CW_{3,8} $ with $Z_4\oplus  Z_2\oplus Z_6$}
\end{figure}
%%%%%%%%%%%%%%%%%%%%%%%%%%%%
%%%%%%%%%%%%%%%%%%%%%%%%%%%%
%%%%%%%%%%%%%%%%%%%%%%%%%%%%
%%%%%%%%%%%%%%%%%%%%%%%%%%%%
%%%%%%%%%%%%%%%%%%%%%%%%%%%%

Recall that Seoud and Youssef~\cite{Seoud-Youssef} proved that closed webs $CW_{2,n}$ are harmonious for all odd $n$. A special case of the previous theorem generalizes their result as follows.

\begin{cor}\label{cor:harm-closed-webs}
	The closed web $CW_{m,n}$ is harmonious for all odd $n\geq3$ and any $m\geq2$.
\end{cor}

We again express the previous result of Theorem~\ref{thm:closed_web} in terms of the nested form of $\Gamma$.

\begin{cor}\label{cor:odd-closed-webs}
	Let $n\geq3$ be odd, $m\geq2$ and $\Gamma$ an Abelian group of order $2mn$ with a cyclic subgroup of order $2m$. Then the closed web $CW_{m,n}$ is $\Gamma$-harmonious.
\end{cor}
%\marbl{Corollary 7.4: What about $CW_{4,5}$ with $\Gamma=Z_{20}+Z_2$?  In this example. $\Gamma$ does not have a cyclic subgroup group of order $2m=8$}
\begin{proof}
	Similarly as in the proof of Corollary~\ref{cor:odd-prisms}, this follows from Theorems~\ref{thm:harm-groups},~\ref{thm:all-cycle}, and ~\ref{thm:closed_web}.
\end{proof}

For even $n$, the explicit expression is broken into several cases.

\begin{cor}\label{cor:even-closed-webs}
	Let $n\geq4$ be even, $m\geq2$ and $\Gamma$ an Abelian group of order $2mn$ written in the nested form as $\Gamma=Z_{n_1}\oplus Z_{n_2}\oplus\dots\oplus Z_{n_t}$ where $t\geq2$, with a cyclic subgroup of order $2m$. If

	\begin{itemize}
	\item[(1)] 	$n_1\equiv0\pmod{8m}$ and $n_2\equiv0\pmod{2}$, or
	\item[(2)] 	$n_1\equiv0\pmod{4m}$, $n_2\equiv0\pmod{2}$, $n_2\geq4$, or
	\item[(3)] 	$n_1\equiv0\pmod{2m}$, $n_2\equiv0\pmod{4}$  and $n_3\equiv0\pmod{2}$, or
	\item[(4)] 	$n_1\equiv0\pmod{2m}$, $n_3\equiv0\pmod{2}$ and $n_t>1$ is odd, 	
	\end{itemize}
	then the closed web $CW_{m,n}$ is $\Gamma$-harmonious.
\end{cor}

\begin{proof} 
	We again denote $n'_1=n_1/(2m)$. By similar reasoning as in the proof of Corollary~\ref{cor:even-prisms}, $\Gamma$ has a subgroup $H\approx Z_{n'_1}\oplus Z_{n_2}\oplus\dots\oplus Z_{n_t}$.
	
	In Case (1),  because $n_1\equiv0\pmod{8m}$, $n'_1\equiv0\pmod{4}$ and $H$ has a subgroup isomorphic to $Z_4\oplus Z_2$. Indeed, $\Gamma/H$ is cyclic of order $2m$. 	
	
	In Case (2), because $n_1\equiv0\pmod{4m}$, $n'_1\equiv0\pmod{2}$ and $H$ has a subgroup isomorphic to $Z_2\oplus Z_{2s}$ with $s\geq 2$ because $n_2\geq4$.

	In Case (3), because $n_2\equiv0\pmod{4}$  and $n_3\equiv0\pmod{2}$, $H$ has a subgroup isomorphic to $Z_4\oplus Z_2$. 
	
	Finally, in Case (4), $n_2$ is even because $n_3$ is even. Because $n_t>1$ is odd, we must have $n_3=2s$ where $s\geq n_t\geq3$, because $n_t$ divides $n_3$.

	Therefore, the proof again follows from Theorems~\ref{thm:harm-groups},~\ref{thm:all-cycle} and~\ref{thm:closed_web}.
\end{proof}

A \emph{generalized open web} $OW_{m,n}$ is the graph arising from the generalized closed web $CW_{m,n}$ by removing the edges of the bottom $n$-cycle $C^m$. The number of edges and order of $\Gamma$ is then $(2m-1)n$.

	The proof of the following theorem is just a slight modification of the proof of Theorem~\ref{thm:closed_web} and is left to the reader.

\begin{thm}\label{thm:open_web} 
	Let $n\geq3,m\geq2$ and $\Gamma$ be an Abelian group of order $(2m-1)n$. If $\Gamma$ has a subgroup $H$ of order $n$ that satisfies the assumptions of Theorem~\ref{thm:all-cycle} and the quotient group $\Gamma/H$ is cyclic, 
	then the generalized open web $OW_{m,n}$ admits a $\Gamma$-harmonious labeling. 
\end{thm}

%%%%%%%%%%%%%%%%%%%%%%%%%%%%
%%%%%%%%%%%%%%%%%%%%%%%%%%%%
%%%%%%%%%%%%%%%%%%%%%%%%%%%%
%%%%%%%%%%%%%%%%%%%%%%%%%%%%
%\newpage
\begin{exm}\label{exm:OW_3,5}
	A $Z_5\oplus Z_5$-harmonious labeling of the open web $OW_{3,5}$.
\end{exm}

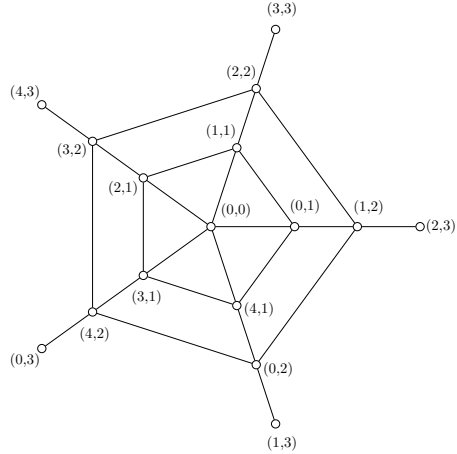
\begin{figure}[h!]\label{fig:OW_{3,5}}
	\centering
	
	\scalebox{0.55}{\begin{tikzpicture}
    % Define the radius for the inner and layertwo pentagons
    \def\Radiusone{2.0}
    \def\Radiusoneinv{2.3}
    \def\Radiustwo{3.5}
    \def\Radiustwoinv{3.8}
    \def\Radiusthree{5}
    \def\Radiusthreeinv{5.5}  
 \def\labelsone{{"(0,1)", "(1,1)", "(2,1)", "(3,1)", "(4,1)"}}
\def\labelstwo{{"(1,2)", "(2,2)", "(3,2)", "(4,2)","(0,2)" }}
\def\labelsthree{{"(2,3)", "(3,3)", "(4,3)", "(0,3)", "(1,3)"}}

    % Draw the inner pentagon
    \foreach \i in {0,1,2,3,4} {
        % Inner pentagon vertices
        \node[draw, circle, inner sep=2pt] (layerone\i) at (72*\i:\Radiusone cm) {};
    }
    \draw (layerone0) -- (layerone1) -- (layerone2) -- (layerone3) -- (layerone4) -- (layerone0);

\foreach \i in {0,1,2,3,4} {
        % Middle pentagon invisible vertices with custom labels
        \node at (72*\i+12:\Radiusoneinv cm) {\pgfmathparse{\labelsone[\i]}\pgfmathresult};
    }

    % Draw the layertwo pentagon
    \foreach \i in {0,1,2,3,4} {
        % Outer pentagon vertices
        \node[draw, circle, inner sep=2pt] (layertwo\i) at (72*\i:\Radiustwo cm) {};
    }
    \draw (layertwo0) -- (layertwo1) -- (layertwo2) -- (layertwo3) -- (layertwo4) -- (layertwo0);
\foreach \i in {0,1,2,3,4} {
        % Middle pentagon invisible vertices with custom labels
        \node at (72*\i+7:\Radiustwoinv cm) {\pgfmathparse{\labelstwo[\i]}\pgfmathresult};
    }

     % Draw the inner pentagon
    \foreach \i in {0,1,2,3,4} {
        % Inner pentagon vertices
        \node[draw, circle, inner sep=2pt] (layerthree\i) at (72*\i:\Radiusthree cm) {};
    }
   
\foreach \i in {0,1,2,3,4} {
        % Middle pentagon invisible vertices with custom labels
        \node at (72*\i:\Radiusthreeinv cm) {\pgfmathparse{\labelsthree[\i]}\pgfmathresult};
    }

    \foreach \i in {0,1,2,3,4} {
        \draw (layerone\i) -- (layertwo\i);
         \draw (layertwo\i) -- (layerthree\i);
        
    }
    \node[draw, circle, inner sep=2pt] (center) at (0,0) {};
    \node (centerlabel) at (0.6,0.4) {(0,0)};
 % Connect the central node to the cycle nodes
    \foreach \i in {0,1,2,3,4} {
        \draw (center) -- (layerone\i);
    }
    
\end{tikzpicture}

} 
	\caption{Labeling of $OW_{3,5} $ with $Z_5\oplus  Z_5$}
\end{figure}
%%%%%%%%%%%%%%%%%%%%%%%%%%%%
\begin{exm}\label{exm:OW_3,8}
	A $Z_4\oplus Z_2\oplus Z_5$-harmonious labeling of the open web $OW_{3,5}$.
\end{exm}

\begin{figure}[h!]\label{fig:OW_{3,8}}
	\centering
	
	\scalebox{0.55}{
\begin{tikzpicture}
    % Define the radius for the inner and layertwo octagons
        \def\Radiusone{2.0}
    \def\Radiusoneinv{2.4}
    \def\Radiustwo{3.5}
    \def\Radiustwoinv{3.9}
    \def\Radiusthree{5}
    \def\Radiusthreeinv{5.7}

    % Define labels for each layer (if needed)
    \def\labelsone{{"(0,0,1)", "(2,0,1)", "(1,0,1)", "(2,1,1)", "(3,0,1)", "(3,1,1)", "(1,1,1)", "(0,1,1)"}}
    
    \def\labelstwo{{"(2,0,2)", "(1,0,2)", "(2,1,2)", "(3,0,2)", "(3,1,2)", "(1,1,2)", "(0,1,2)","(0,0,2)" }}
   \def\labelsthree{{ "(1,0,3)", "(2,1,3)", "(3,0,3)", "(3,1,3)", "(1,1,3)", "(0,1,3)","(0,0,3)", "(2,0,3)"}}

    % Draw the inner octagon
    \foreach \i in {0,...,7} {
        % Inner octagon vertices
        \node[draw, circle, inner sep=2pt] (layerone\i) at (45*\i:\Radiusone cm) {};
    }
    \draw (layerone0) -- (layerone1) -- (layerone2) -- (layerone3) -- (layerone4) -- (layerone5) -- (layerone6) -- (layerone7) -- (layerone0);

    % Invisible vertices with custom labels for the inner octagon
    \foreach \i in {0,...,7} {
        \node at (45*\i+15:\Radiusoneinv cm) {\pgfmathparse{\labelsone[\i]}\pgfmathresult};
    }

    % Draw the layertwo octagon
    \foreach \i in {0,...,7} {
        % layertwo octagon vertices
        \node[draw, circle, inner sep=2pt] (layertwo\i) at (45*\i:\Radiustwo cm) {};
    }
    \draw (layertwo0) -- (layertwo1) -- (layertwo2) -- (layertwo3) -- (layertwo4) -- (layertwo5) -- (layertwo6) -- (layertwo7) -- (layertwo0);

    % Invisible vertices with custom labels for the layertwo octagon
    \foreach \i in {0,...,7} {
        \node at (45*\i+10:\Radiustwoinv cm) {\pgfmathparse{\labelstwo[\i]}\pgfmathresult};
    }

    % Draw the layerthree octagon
    \foreach \i in {0,...,7} {
        % layerthree octagon vertices
        \node[draw, circle, inner sep=2pt] (layerthree\i) at (45*\i:\Radiusthree cm) {};
    }

    % Invisible vertices with custom labels for the layerthree octagon
    \foreach \i in {0,...,7} {
        \node at (45*\i:\Radiusthreeinv cm) {\pgfmathparse{\labelsthree[\i]}\pgfmathresult};
    }

    % Connect vertices between layers
    \foreach \i in {0,...,7} {
        \draw (layerone\i) -- (layertwo\i);
        \draw (layertwo\i) -- (layerthree\i);
    }

    \node[draw, circle, inner sep=2pt] (center) at (0,0) {};
    \node (centerlabel) at (0.94,0.3) {(0,0,0)};
 % Connect the central node to the cycle nodes
    \foreach \i in {0,1,2,3,4,5,6,7} {
        \draw (center) -- (layerone\i);
    }
\end{tikzpicture}
} 
	\caption{Labeling of $OW_{3,8} $ with $Z_4\oplus  Z_2\oplus Z_5$}
\end{figure}
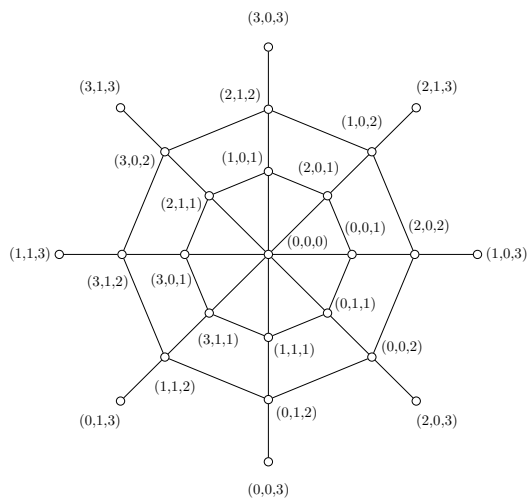

%%%%%%%%%%%%%%%%%%%%%%%%%%%%

%%%%%%%%%%%%%%%%%%%%%%%%%%%%
%%%%%%%%%%%%%%%%%%%%%%%%%%%%
%%%%%%%%%%%%%%%%%%%%%%%%%%%%
%%%%%%%%%%%%%%%%%%%%%%%%%%%%
%%%%%%%%%%%%%%%%%%%%%%%%%%%%

Seoud and Youssef~\cite{Seoud-Youssef} and Gnanajothi~\cite{Gnanajothi} proved that open webs $OW_{m,n}$ are harmonious for all odd $n$ and $m=2$ and 3, respectively. We state a generalization following directly from the previous theorem as a corollary.

\begin{cor}\label{cor:harm-open-webs}
	The open web $OW_{m,n}$ is harmonious for all odd $n$ and any $m\geq2$.
\end{cor}

For convenience, we again express the result of Theorem~\ref{thm:open_web} in a more explicit form, describing the group $\Gamma$ in the nested form.

\begin{cor}\label{cor:odd-open-webs}
	Let $n\geq3$ be odd, $m\geq2$ and $\Gamma$ an Abelian group of order $(2m-1)n$ with a cyclic subgroup of order $2m-1$. Then the open web $OW_{m,n}$ admits a $\Gamma$-harmonious labeling.    
\end{cor}

\begin{proof}
	This follows directly from Theorems~\ref{thm:harm-groups},~\ref{thm:all-cycle} and ~\ref{thm:open_web}.
\end{proof}

\begin{cor}\label{cor:even-open-webs}
	Let $n\geq4$ be even, $m\geq2$ and $\Gamma$ an Abelian group of order $(2m-1)n$  with a cyclic subgroup of order $2m-1$ written in the nested form as $\Gamma=Z_{n_1}\oplus Z_{n_2}\oplus\dots\oplus Z_{n_t}$, where $t\geq2$. If
	\begin{itemize}
	\item[(1)] 	$n_1\equiv0\pmod4$ and $n_2$ is even, or 
	\item[(2)]	$n_2\geq4$ is even,
	\end{itemize}
	then the web $OW_{m,n}$ admits a $\Gamma$-harmonious labeling.    
\end{cor}

The proofs are identical to the proofs of Corollaries~\ref{cor:odd-prisms} and~\ref{cor:even-prisms}, respectively.

\section{Conclusion}\label{sec:conclusion}

We explored $\Gamma$-harmonious labeling, a generalization of a well-known notion of harmonious labeling, for several families of cycles-related graphs. In most cases, our results are only partial. Therefore, we state some open problems here.

There are two main directions for further research. One is dealing with graphs based on even cycles where the group $\Gamma$ does not contain a subgroup isomorphic to $Z_4\oplus Z_2$. 

\begin{prb}\label{prb:even-cycles}
	Determine for what Abelian groups $\Gamma$ not containing a subgroup of order $n$ isomorphic to $Z_2\oplus Z_{2s}, s\geq2$ there exists a $\Gamma$-harmonious labeling of generalized prisms $Y_{m,n}$, open and closed webs $OW_{m,n}$, $CW_{m,n}$ and superwheels $SW_{k,n}$ for even $n$.
\end{prb}

For Dutch windmill graphs, we only covered  the cases when the number of cycles is odd. When the cycle length was even, we have an additional restriction on the quotient group. We propose the following two problems on Dutch windmill graphs.

\begin{prb}\label{prb:windmill1}
	Determine for what Abelian groups $\Gamma$ there exists a $\Gamma$-harmonious labeling of Dutch windmill graphs $D^m_{n}$ for even $m$.
\end{prb}

\begin{prb}\label{prb:windmill2}
	Determine for what Abelian groups $\Gamma$ there exists a $\Gamma$-harmonious labeling of Dutch windmill graphs $D^m_{n}$ for odd $m$ and even $n$ when there does not exist a subgroup $ H$ of order $m$ such that $K=\Gamma/H$ satisfies the conditions in Theorem \ref{thm:all-cycle}.
\end{prb}

The other direction concerns graphs $Y_{m,n}$ and $OW_{m,n}$ when $\Gamma$ does not have a cyclic subgroup of order $2m-1$ or $CW_{m,n}$ when $\Gamma$ does not have a cyclic subgroup of order $2m$.

\begin{prb}\label{prb:no-cyclic}
	Determine for what Abelian groups $\Gamma$ not containing a cyclic subgroup 
	\begin{enumerate}
	\item 
		of order $2m-1$ there exists a $\Gamma$-harmonious labeling of generalized prisms $Y_{m,n}$ and open webs $OW_{m,n}$
	\item 
		of order $2m$ there exists a $\Gamma$-harmonious labeling of closed webs $CW_{m,n}$.
	\end{enumerate}
\end{prb}

We also want to express out thanks to Alexa Hedtke and Johnathan Koch for their contribution to our results in the early stages of this project.
Johnathan Koch wrote a program in python which we used to check if a $\Gamma$-harmonious labeling existed for a given graph and if such labeling existed, it gave us the exact labeling.
The program helped us to find counterexamples and investigate patterns. The program can be found here \cite{John}.

\vskip1cm

\noindent

\newpage

%\newpage

\end{document}